\documentclass[11pt,reqno]{amsart}

\usepackage{enumerate,tikz-cd,mathtools,amssymb}

\usepackage[a4paper,top=2.8cm,bottom=2.5cm,left=2.9cm,right=2.9cm,marginparwidth=1.75cm]{geometry}

\usepackage{verbatim}
\usepackage{mathrsfs}
\usepackage{array, makecell}
\usepackage{float}
\usepackage{longtable}
\usepackage{graphicx}
\usepackage{caption}
\usepackage[mathscr]{euscript}
\usepackage{upgreek}
\usepackage{libertine}
\usepackage{xcolor}
\usepackage{tikz}
\usepackage{color}
\usepackage[all]{xy}
\usepackage{caption}
\usetikzlibrary{matrix,arrows,decorations.pathmorphing}
\usepackage{courier}
\usepackage{listings}
\allowdisplaybreaks
\usepackage{microtype}
\usepackage{longtable}
\usepackage{tabularx}
\usepackage{multicol}

\usepackage{subcaption}

\definecolor{commentgreen}{RGB}{2,112,10}
\definecolor{eminence}{RGB}{108,48,130}
\definecolor{frenchplum}{RGB}{149,20,83}

\definecolor{ffqqqq}{RGB}{215,25,28}
\definecolor{cccccc}{RGB}{171,217,233}
\definecolor{cite}{RGB}{44,123,182}
\definecolor{ref}{RGB}{215,25,28}

\lstdefinelanguage{Sage}
{
keywords={load, matrix, from, import},
emph={},
}

\usepackage[pdftex,
              pdfauthor={Riccardo Moschetti, Franco Rota, Luca Schaffler},
              pdftitle={},
              pdfsubject={non-degeneracy invariant, Enriques surfaces, computational view},
              pdfkeywords={Enriques surface, elliptic fibration, rational curve, non-degeneracy invariant, Fano polarization, Moschetti,  Rota, Schaffler},
              colorlinks=true,
              citecolor=cite,
              linkcolor=ref 
              ]{hyperref}

 \usetikzlibrary{shapes}

\tikzstyle{dot}=[circle,draw,fill=black,inner sep=0pt, minimum width=4pt]
\tikzstyle{wdot}=[circle,draw,fill=white,inner sep=0pt, minimum width=4pt]
\tikzstyle{fiber}=[rectangle,draw=black,thick,inner sep=0pt, minimum width=6pt, minimum height=6pt,preaction={draw=gray!6, line width=3pt}]
\tikzstyle{half-fiber}=[dashed,rectangle,draw=black,thick,inner sep=0pt, minimum width=6pt, minimum height=6pt]

\tikzstyle{doubleline}=[rectangle,draw=red,thick,inner sep=0pt, minimum width=6pt, minimum height=6pt,preaction={draw=gray!6, line width=3pt}]

\tikzset{double distance = 2pt}
 
\tikzset{
    partial ellipse/.style args={#1:#2:#3}{
        insert path={+ (#1:#3) arc (#1:#2:#3)}
    }
}

\makeatletter
\newtheorem*{rep@theorem}{\rep@title}
\newcommand{\newreptheorem}[2]{
\newenvironment{rep#1}[1]{
 \def\rep@title{#2 \ref{##1}}
 \begin{rep@theorem}}
 {\end{rep@theorem}}}
\makeatother

\theoremstyle{plain}
\newtheorem{theorem}{Theorem}[section]
\newtheorem{corollary}[theorem]{Corollary}
\newtheorem{lemma}[theorem]{Lemma}
\newtheorem{proposition}[theorem]{Proposition}

\newtheorem*{lemma*}{Lemma}

\theoremstyle{definition}
\newtheorem{definition}[theorem]{Definition}
\newtheorem{example}[theorem]{Example}
\newtheorem{remark}[theorem]{Remark}
\newtheorem{question}[theorem]{Question}
\newtheorem*{claim*}{Claim}
\newtheorem*{ack*}{Acknowledgements}

\newcommand{\nd}{\mathrm{nd}}
\newcommand{\Fnd}{\mathrm{Fnd}}
\newcommand{\Mnd}{\mathrm{Mnd}}

\newcommand{\ttbar}{(\tau,\overline{\tau})}

\newcommand{\One}{\sfO_2}
\newcommand{\Oext}{\sfO_1}

\newcommand{\Spec}{\mathrm{Spec}}

\DeclareMathOperator{\Num}{Num}
\DeclareMathOperator{\Pic}{Pic}

\DeclareMathOperator{\Aut}{Aut}
\DeclareMathOperator{\aut}{aut}

\newcommand{\pr}[1]{\mathbb P^{#1}}

\newcommand{\f}{\mathrm{F}}
\newcommand{\hf}{\mathrm{HF}}
\newcommand{\Nef}{\mathrm{Nef}}
\newcommand{\wnod}{W^\mathrm{nod}}

\usepackage[mathscr]{euscript}

\docsvlist{A,B,C,D,E,F,G,H,I,J,K,L,M,N,O,P,Q,R,S,T,U,V,W,X,Y,Z} 

\docsvlist{A,B,C,D,E,F,G,H,I,J,K,L,M,N,O,P,Q,R,S,T,U,V,W,X,Y,Z} 

\docsvlist{A,B,C,D,E,F,G,H,I,J,K,L,M,N,O,P,Q,R,S,T,U,V,W,X,Y,Z} 

\docsvlist{A,B,C,D,E,F,G,H,I,J,K,L,M,N,O,P,Q,R,S,T,U,V,W,X,Y,Z} 

\docsvlist{A,B,C,D,E,F,G,H,I,J,K,L,M,N,O,P,Q,R,S,T,U,V,W,X,Y,Z,
a,b,c,d,e,f,g,h,i,j,k,l,m,n,o,p,q,r,s,t,u,v,w,x,y,z} 

\docsvlist{A,B,C,D,E,F,G,H,I,J,K,L,M,N,O,P,Q,R,S,T,U,V,W,X,Y,Z,
a,b,c,d,e,f,g,h,i,j,k,l,m,n,o,p,q,r,s,t,u,v,w,x,y,z} 

\begin{document}
\title[Enriques surfaces with non-generic non-degeneracy]{
Enriques surfaces with non-generic non-degeneracy 
}

\author[R. Moschetti]{Riccardo Moschetti}
\address{RM: Department of Mathematics F. Casorati, University of Pavia, via Ferrata 5, 27100 Pavia, Italy} 
\email{riccardo.moschetti@unipv.it}

\author[F. Rota]{Franco Rota}
\address{FR: Universit\'e Paris-Saclay, Laboratoire de math\'ematiques d'Orsay, Rue Michel Magat, 91405 Orsay, France} 
\email{franco.rota@universite-paris-saclay.fr}

\author[L. Schaffler]{Luca Schaffler}
\address{LS: Dipartimento di Matematica e Fisica, Universit\`a degli Studi Roma Tre, Largo San Leonardo Murialdo 1, 00146, Roma, Italy}
\email{luca.schaffler@uniroma3.it}


\begin{abstract}
We study the non-degeneracy invariant $\mathrm{nd}(Y)$ of complex Enriques surfaces in families. Our first main result shows that $\mathrm{nd}(Y)$ cannot increase under specialization. The second main result is the conclusion of the computation of the non-degeneracy invariant for the $155$ families of $(\tau,\overline{\tau})$-generic surfaces introduced by Brandhorst and Shimada. Of the previously known $144$ cases, only $3$ satisfy $\mathrm{nd}(Y)\neq10$, which is the non-degeneracy invariant of a general Enriques surface. The remaining $11$ families studied in this article also have non-generic non-degeneracy. To compute this, we produce upper bounds on $\mathrm{nd}(Y)$ by refining this invariant into two others: the Fano and Mukai non-degeneracy invariants, which are related to two different classes of projective realizations of Enriques surfaces. As a result, we find the first known examples of Enriques surfaces with $\mathrm{nd}(Y)=9$.
\end{abstract}

\maketitle

\section{Introduction}

An Enriques surface is a smooth minimal algebraic surface of Kodaira dimension zero with vanishing irregularity and geometric genus. Enriques surfaces are elliptically fibered, and the geometry and combinatorics of these fibrations encode meaningful geometric information. To capture this, in \cite{CD89} Cossec and Dolgachev introduced the so-called \emph{non-degeneracy invariant}. More precisely, an elliptic pencil on $Y$ has exactly two double fibers, whose underlying reduced curves are called \emph{half-fibers}. The non-degeneracy invariant $\nd(Y)$ is defined to be the maximum number of half-fibers $F_1,\ldots,F_m$ on $Y$ such that $F_i \cdot F_j = 1-\delta_{ij}$ for all $i,j\in\{1,\ldots,m\}$.
The invariant $\nd(Y)$ is closely related to certain, possibly singular, projective realizations of $Y$: its relation with the Fano models of Enriques surfaces is classical, but it will be relevant for this paper to investigate its ties with another class of projective models introduced by Mukai. We review this in Section~\ref{sec:Fano-Mukai-polarizations}, where we also give references.
The importance of this invariant is reflected also in its role in decomposing the derived category of $Y$ \cite{LSZ22}.

The value of $\nd(Y)$ for general Enriques surfaces is known, but the invariant is hard to compute for special ones. For an overview of the known cases, see the Introduction of \cite{MRS24Paper}. For special Enriques surfaces, it is natural to look first for estimates of $\nd(Y)$. To bound $\nd(Y)$ from below it suffices to exhibit appropriate sequences of half-fibers, and this was exploited for instance in \cite{MRS22Paper,MRS24Paper}. It is much more difficult to produce upper bounds. This is achieved, in a sense, in the classification of surfaces of non-degeneracy $3$ of \cite{MMV22_NonDeg3} and in the explicit computation in \cite[Section~5]{MRS24Paper}. The behavior of $\nd(Y)$ in families is also not fully understood yet: under specialization, on the one hand the difference $F_i-F_j$ between two half-fibers may become effective, lowering the non-degeneracy invariant; on the other hand, the nef cone $\Nef(Y)$ may change, giving rise to new elliptic fibrations, potentially increasing this invariant.

\subsection{Main results}
The first main result of this paper controls the non-degeneracy invariant under specialization in a family of Enriques surfaces.

\begin{theorem}[Corollary~\ref{cor:family}]\label{thm:family_intro}
Let $Y\to T$ be a family of Enriques surfaces with $T$ an irreducible and Noetherian scheme, and let $0\in T$ be a closed point. Then, for a very general $t \in T$ we have that
\[
\nd(Y_0)\leq \nd(Y_t).
\]
\end{theorem}

We apply Theorem~\ref{thm:family_intro} to the class of $\ttbar$-generic Enriques surfaces (Definition~\ref{def:tautaubar-generic-Enriques-surfaces}) of Brandhorst and Shimada \cite{BS22}. These are, roughly speaking, Enriques surfaces $Y$ containing a configuration of smooth rational curves which span a lattice of ADE-type $\tau$ with primitive closure of ADE-type $\overline{\tau}$ in $\Num(Y)$, the group of divisors modulo numerical equivalence. The results in \cite{BS22} show that, for fixed $\tau$ and $\overline{\tau}$, all $\ttbar$-generic Enriques surfaces have the same elliptic fibrations and hence the same non-degeneracy invariant.
Building upon this theory, in Definition~\ref{def:ttbar_surface} we introduce \emph{$(\tau,\overline{\tau})$ Enriques surfaces}, which can be thought of as limits of $(\tau,\overline{\tau})$-generic Enriques surfaces, and obtain the following upper bounds for their non-degeneracy invariant.

\begin{theorem}[Theorem~\ref{thm:boundLatticePol}]\label{thm:LatticePol_Intro}
    Let $Y$ be a $\ttbar$ Enriques surface. 
Then, $\nd(Y)\leq \nd(\overline{Y})$, where $\overline{Y}$ is any $\ttbar$-generic Enriques surface.
\end{theorem}

This result allows to compute the non-degeneracy invariant of every Enriques surface that contains a configuration of smooth rational curves with dual graph as in \eqref{fig:configurations-smooth-rational-curves-BP-MLP1}~(a) or (b). See Remark~\ref{rmk:description} for other characterizations of these surfaces.

\begin{equation}
\noindent
\minipage{0.5\textwidth}
\centering
\begin{tikzpicture}[scale=0.5]

\node (a) at (180:4)[]{$(a)$};

\node (F1) at (-135:2) [dot, label=left:{}]{};
\node (F2) at (-90:2) [dot, label=below:{}]{};
\node (F3) at (-45:2) [dot, label=right:{}]{};
\node (F4) at (0:2) [dot, label=right:{}]{};
\node (F5) at (45:2) [dot, label=right:{}]{};
\node (F6) at (90:2) [dot, label=above:{}]{};
\node (F7) at (135:2) [dot, label=left:{}]{};
\node (F8) at (180:2) [dot, label=left:{}]{};
\node (L1) at (180:1) [dot, label=above:{}]{};
\node (N) at (0:1) [dot, label=above:{}]{};

\draw[] (F2)--(F3)--(F4)--(F5)--(F6)--(F7)--(F8) (F4)--(N);
\draw[]  (F8)--(L1) (F2)--(F1)--(F8);

\end{tikzpicture}
\endminipage
\minipage{0.40\textwidth}
\centering
\begin{tikzpicture}[scale=0.5]

\node (b) at (180:4)[]{$(b)$};

\node (1) at (90:2) [dot]{};
\node (2) at (45:2) [dot,label=right:{}]{};
\node (3) at (45:1) [dot,label=below:{}]{};
\node (4) at (0:2) [dot,label=right:{}]{};
\node (5) at (-45:2) [dot,label=right:{}]{};
\node (6) at (-45:1) [dot]{};
\node (7) at (-90:2) [dot,label=below:{}]{};
\node (8) at (-135:2) [dot,label=left:{}]{};
\node (9) at (-135:1) [dot]{};
\node (10) at (180:2) [dot,label=left:{}]{};
\node (11) at (135:2) [dot,label=left:{}]{};
\node (12) at (135:1) [dot]{};

\draw[] (2)--(4)--(3) (4)--(5)--(7)--(8)--(10)--(11);
\draw[]  (3)--(1)--(2) (4)--(6)--(7)--(9)--(10)--(12)--(1)--(11);

\end{tikzpicture}
\endminipage
\label{fig:configurations-smooth-rational-curves-BP-MLP1}
\end{equation}

\begin{theorem}[Corollary \ref{cor:ndBP} and \ref{cor:ndmlp1}]
Let $Y$ be an Enriques surface containing a configuration of smooth rational curves with dual graph (a) (resp., (b)) as in \eqref{fig:configurations-smooth-rational-curves-BP-MLP1}. Then, $\nd(Y)=4$ (resp., $\nd(Y)=8$).
\end{theorem}

Our second main result, independent of Theorem \ref{thm:family_intro} but also a key input for Theorem \ref{thm:LatticePol_Intro}, is the completion of the computation, started in \cite{MRS24Paper}, of the non-degeneracy invariant for the 155 families of $\ttbar$-generic Enriques surfaces.  For completeness, we state here the value of the non-degeneracy invariant for all the $\ttbar$-generic Enriques surfaces. Except for $i = 145,172,184$, all the cases in \eqref{eq:14cases_Intro} are new and substantially harder than many of the previously known cases, since they require the computation of an upper bound for their non-degeneracy invariant. Earlier partial results are recalled in Section~\ref{sec:mainttb}.

\begin{theorem}[Theorem~\ref{thm:ndTauBarTau}]
\label{thm:TauBarTauFinal}
Let $Y_i$ denote the $i$-th $\ttbar$-generic Enriques surface as listed in \cite[Table~1]{BS22}. Then $\nd(Y_i)=10$, except for the following cases:
\begin{equation}\label{eq:14cases_Intro}
\begin{tabular}{ccccc}
$\nd(Y_{84})=9$, & $\nd(Y_{85})=7$, & $\nd(Y_{121})=9$, & $\nd(Y_{122})=7$, & $\nd(Y_{123})=7$, \\
$\nd(Y_{143})=8$, & $\nd(Y_{144})=8$, & $\nd(Y_{145})=4$, & $\nd(Y_{158})=9$, & $\nd(Y_{159})=7$, \\
$\nd(Y_{171})=8$, & $\nd(Y_{172})=4$, & $\nd(Y_{176})=7$, & $\nd(Y_{184})=7$. & 
\end{tabular}
\end{equation}
\end{theorem}

We remark that the $(\tau,\overline{\tau})$-generic Enriques surfaces $Y_{84},Y_{121},Y_{158}$ are the first known examples of Enriques surfaces with non-degeneracy invariant $9$. We are not aware of any example of Enriques surfaces with non-degeneracy invariant $5$ or $6$. Next, we discuss some of the technical ideas behind the proofs of these results.

\subsection{The non-degeneracy invariant in families}

Theorem \ref{thm:family_intro} is based on the following observation of independent interest.

\begin{theorem}[Theorem \ref{thm:family}]
    Let $Y\to T$ be a smooth projective morphism of schemes over $\bC$ of relative dimension $2$. Suppose that $T$ is irreducible and Noetherian and that for a closed point $0\in T(\bC)$, the fiber $Y_0$ satisfies $H^1(Y_0, \sO_{Y_0}) = H^2(Y_0, \sO_{Y_0}) = 0$. Then, for very general $t\in T(\bC)$ there is an identification $\Pic(Y_0)\cong \Pic(Y_t)$ which preserves intersections and induces an inclusion $\Nef(Y_0) \subseteq \Nef(Y_t)$.
\end{theorem}

The proof leverages the argument of \cite[Theorem~4.1]{TotaroNef}, which implies that, suitably locally, the restriction maps from the relative Picard scheme of $Y\to T$ to the closed and geometric generic fibers are both isomorphisms. We then show that this identification preserves intersection products, thus controlling the nef cones.
This allows to transfer non-degenerate sequences from the special to the geometric generic fiber, which implies Theorem \ref{thm:family_intro}.

To apply Theorem \ref{thm:family_intro} in the setting of Theorem \ref{thm:LatticePol_Intro}, consider first a $\ttbar$-generic surface $Y$, and let $X$ denote its covering K3. Then, $\Num(Y)$ contains an ADE lattice $R$ of type $\tau$ and primitive closure of type $\overline{\tau}$ such that $\Num(X)$ is isomorphic to a specific sublattice $M_R$ of $L_{10}(2)\oplus R(2)$. 
Essentially, we define a $\ttbar$ Enriques surface as an Enriques surface whose covering K3 is $M_R$-polarized (in the sense of \cite{Dol96,AE25}). Moduli stacks of polarized K3 surfaces admit universal families $\sX = \sX(M_R)$. There is an involution on $\sX$ which produces a family of Enriques surfaces $\sY$. We apply Theorem \ref{thm:family_intro} to $\sY$ and obtain Theorem \ref{thm:LatticePol_Intro}.

\subsection{Computing the non-degeneracy invariant: Fano and Mukai models}

We have that $\nd(Y)$ is tied to two classes of nef divisors on an Enriques surface $Y$, called \emph{Fano} and \emph{Mukai} polarizations. To a Fano or a Mukai polarization $D$ corresponds a unique isotropic sequence, i.e. a collection $\bff \coloneqq (f_1,\ldots,f_n)$ of isotropic vectors of $\Num(Y)$ satisfying $f_i\cdot f_j = 1-\delta_{ij}$ (Lemma~\ref{lem_UniqueCanonicalFano} and \ref{lem_UniqueCanonicalMukai}). The number of half-fiber classes appearing in $\bff$ is called the \textit{non-degeneracy of} $\bff$.
Then, we define the \emph{Fano non-degeneracy invariant} $\Fnd(Y)$ to be the highest non-degeneracy of isotropic sequence corresponding to a Fano polarization (Definition~\ref{def_Fnd}). Similarly, the \emph{Mukai non-degeneracy invariant} $\Mnd(Y)$ is the highest non-degeneracy of isotropic sequence arising from a Mukai polarization (Definition~\ref{def_Mnd}). The classical $\nd(Y)$ is simply the maximum between $\Fnd(Y)$ and $\Mnd(Y)$ (Corollary~\ref{cor:max}). 

\begin{theorem}
The following hold:
\begin{itemize}

\item[(i)] The results of Theorem~\ref{thm:family_intro}, Theorem~\ref{thm:LatticePol_Intro} hold verbatim if $\nd(-)$ is replaced everywhere with $\Fnd(-)$, or if it is replaced everywhere with $\Mnd(-)$.

\item[(ii)] Let $Y_i$ denote the $i$-th $\ttbar$-generic Enriques surface as listed in \cite[Table~1]{BS22}. Then, $\Fnd(Y_i)=10$ and $\Mnd(Y_i)=9$, except for the cases listed in \eqref{eq:14cases_Intro}, where $\Fnd(Y_i) = \Mnd(Y_i) = \nd(Y_i)$.

\end{itemize}
\end{theorem}

While Fano and Mukai polarizations do not affect substantially the proofs of Theorem~\ref{thm:family_intro} and Theorem~\ref{thm:LatticePol_Intro}, they play a crucial role in the computations of Theorem~\ref{thm:TauBarTauFinal}, which we now outline.

If $D$ is a Fano polarization with associated sequence $\bff$, the non-degeneracy of $\bff$ is precisely $10-d$, where $d$ is the number of smooth rational curves contracted by $D$. An analogous, but more delicate, relation holds for Mukai polarizations (Lemma~\ref{lem_UniqueCanonicalMukai}): for each Mukai polarization there may exist one contracted curve which does not lower the non-degeneracy of the associated sequence. This led us to distinguishing $\Fnd$ and $\Mnd$.

Thus, finding upper bounds for $\Fnd(Y)$ and $\Mnd(Y)$ translates to showing that all Fano or Mukai polarizations contract sufficiently many smooth rational curves (Theorem~\ref{cor_ndUpperBound}). 
Using \cite{BS22}, for each $\ttbar$-generic Enriques surface in \eqref{eq:14cases_Intro} we proceed as follows:

\begin{itemize}
    \item Produce a list of representatives of Fano and Mukai polarizations modulo the action of the automorphism group. This is possible since a fundamental domain of the action of $\Aut(Y)$ on $\Nef(Y)$ is isometric to the nef cone of an Enriques surface with finite automorphism group, where we have an exhaustive list of Fano and Mukai polarizations.
    \item Intersect each representative Fano or Mukai polarization with smooth rational curves, to produce enough contracted curves and obtain the desired upper bound.
\end{itemize}

We conclude the paper with Section~\ref{sec:Mukai}, where we give examples of non-ample Mukai polarizations of maximal non-degeneracy $9$ (in contrast, a Fano polarization of maximal non-degeneracy $10$ has to be ample), and an example of an Enriques surface with an ample Mukai polarization which admits no ample Fano polarizations.

\subsection*{Acknowledgements}
We would like to thank Simon Brandhorst, Igor Dolgachev, Philip Engel, Shigeyuki Kond\=o, Emanuele Macr\`i, Gebhard Martin, Giacomo Mezzedimi, and Davide Cesare Veniani for helpful conversations.

\subsection*{Fundings}
F. Rota is supported by the European Union’s Horizon 2020 Research and Innovation Programme under the Marie Skłodowska-Curie grant agreement n. 101147384 (\href{https://cordis.europa.eu/project/id/101147384}{CHaNGe}). He acknowledges support from the EPSRC grant EP/R034826/1. 
L. Schaffler was supported by the projects ``Programma per Giovani Ricercatori Rita Levi Montalcini'' and PRIN2020KKWT53 ``Curves, Ricci flat Varieties and their Interactions''. R. Moschetti and L. Schaffler are supported by PRIN 2022 ``Moduli Spaces and Birational Geometry'' – CUP E53D23005790006 and are members of the INdAM group GNSAGA. The development of this paper was partially supported also by the project ``Numerical invariants of nodal Enriques surfaces'', funded by the INdAM group GNSAGA.

\subsection*{Open access} For the purpose of open access, the authors have applied a Creative Commons Attribution (CC:BY) licence to any Author Accepted Manuscript version arising from this submission.

\section{Enriques surfaces, non-degeneracy, Fano and Mukai polarizations}

\subsection{Enriques surfaces and lattices}

Enriques surfaces constitute a fundamental slice of the classification of complex smooth minimal algebraic surfaces of Kodaira dimension zero, and they are characterized by the vanishing of the geometric genus and irregularity. For an Enriques surface $Y$, the group of divisors modulo numerical equlvalence, usually denoted $\mathrm{Num}(Y)$, is endowed with the structure of a (non-degenerate) lattice via the intersection product among curves. For simplicity of notation, we will write $S_Y:=\mathrm{Num}(Y)$. Since distinct smooth rational curves on $Y$ cannot be numerically equivalent, we will identify them with their respective numerical equivalence classes. Independently of the Enriques surface $Y$, the lattice $S_Y$ is isometric to $U\oplus E_8$, where $U$ denotes the hyperbolic plane and $E_8$ is the negative definite lattice associated to the corresponding Dynkin diagram.

Finally, we recall that $U\oplus E_8$ is isometric to the \emph{$L_{10}$ lattice} (also denoted $E_{10}$): this is defined to be $\mathbb{Z}^{10}$ together with the intersection form associated to the canonical basis $e_1,\ldots,e_{10}$ as represented in Figure~\ref{fig:E10}: $e_i^2=-2$ and, for $i\neq j$, $e_i\cdot e_j=1$ if the corresponding vertices are joined by an edge, and zero otherwise. The isometry $L_{10}\cong U\oplus E_8$ follows by \cite[Theorem~1]{Mil58}.

\begin{figure}
\begin{tikzpicture}

\draw  (0,0)-- (1,0);
\draw  (1,0)-- (2,0);
\draw  (2,0)-- (3,0);
\draw  (3,0)-- (4,0);
\draw  (4,0)-- (5,0);
\draw  (5,0)-- (6,0);
\draw  (6,0)-- (7,0);
\draw  (7,0)-- (8,0);
\draw  (2,0)-- (2,1);

\fill [color=black] (0,0) circle (2.4pt);
\fill [color=black] (1,0) circle (2.4pt);
\fill [color=black] (2,0) circle (2.4pt);
\fill [color=black] (3,0) circle (2.4pt);
\fill [color=black] (4,0) circle (2.4pt);
\fill [color=black] (5,0) circle (2.4pt);
\fill [color=black] (6,0) circle (2.4pt);
\fill [color=black] (7,0) circle (2.4pt);
\fill [color=black] (8,0) circle (2.4pt);
\fill [color=black] (2,1) circle (2.4pt);

\draw[color=black] (2.4,1) node {$e_1$};
\draw[color=black] (0,-0.4) node {$e_2$};
\draw[color=black] (1,-0.4) node {$e_3$};
\draw[color=black] (2,-0.4) node {$e_4$};
\draw[color=black] (3,-0.4) node {$e_5$};
\draw[color=black] (4,-0.4) node {$e_6$};
\draw[color=black] (5,-0.4) node {$e_7$};
\draw[color=black] (6,-0.4) node {$e_8$};
\draw[color=black] (7,-0.4) node {$e_9$};
\draw[color=black] (8,-0.4) node {$e_{10}$};

\end{tikzpicture}
\caption{The $L_{10}$ lattice. The labeling of the roots is compatible with \cite{BS22}.}
\label{fig:E10}
\end{figure}
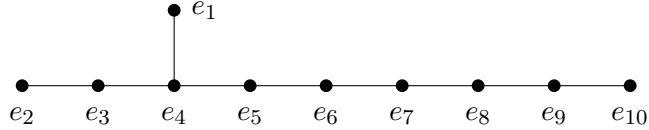

\subsection{Isotropic sequences on Enriques surfaces}

Recall that an isotropic sequence on an Enriques surface $Y$ is a collection of classes $(f_1,\ldots,f_n)$ in $S_Y$ such that $f_i \cdot f_j = 1-\delta_{ij}$ for all $i,j$. Since the vectors $f_i$ are independent, the maximum length of an isotropic sequence is $10$. We say that an isotropic sequence $(f_1,\ldots, f_n)$ is \emph{non-degenerate} if each $f_i$ is nef. The maximum length of non-degenerate isotropic sequences is an important geometric invariant for an Enriques surface.

\begin{definition}
\label{def:nd-invariant}
Let $Y$ be an Enriques surface. The \textit{non-degeneracy} invariant of $Y$, $\nd(Y)$, is the maximum integer $n$ for which there exists a non-degenerate isotropic sequence of length $n$.
\end{definition}

The Weyl group $W(L_{10})$ is the group generated by the reflections across the $(-2)$-vectors of $L_{10}$. Then, $W(L_{10})$ acts on the set of isotropic sequences. More precisely, the following holds.

\begin{proposition}[{\cite[Proposition~6.1.1]{DK25}}]\label{prop_transitivity}
    The Weyl group $W(L_{10})$ acts transitively on the set of isotropic sequences of length $k \neq 9$ and it has exactly two orbits of isotropic sequences of length $9$. 
\end{proposition}

This motivates the introduction of the following notations.

\begin{definition}

\

\begin{itemize}
    \item Let $\Oext$ be the set of length $9$ isotropic sequences which can be extended to an isotropic sequence of length $10$;
    \item Let $\One$ be the set of length $9$ isotropic sequences which cannot be extended to an isotropic sequence of length $10$.
\end{itemize}
\end{definition}

Combining results in \cite{CDL25}, we have that the two orbits $\Oext$ and $\One$ can be distinguished as follows.

\begin{proposition}\label{prop_OrbitsLength9}
Let $\bff = (f_1,\ldots,f_9)$ be an isotropic sequence of length $9$. Then, the following hold:
\begin{itemize}

\item $\bff\in\Oext$ if and only if $\sum_{i=1}^9f_i$ is primitive in $S_Y$;

\item $\bff\in\One$ if and only if $\sum_{i=1}^9f_i$ is divisible by $2$ in $S_Y$.

\end{itemize}
\end{proposition}

\begin{proof}
By Proposition~\ref{prop_transitivity}, it is enough to exhibit an extensible (resp. non-extensible) isotropic sequence of length $9$ such that the sum of its elements is primitive (resp. divisible by $2$).

We have that the sequence $\sff_1,\ldots,\sff_9$ in \cite[Proposition~1.5.3]{CDL25} is extensible and it can be verified that $\sff_1+\ldots+\sff_9$ is primitive directly from the definition of the $\sff_i$. On the other hand, $\sfg_1,\ldots,\sfg_9$ in the proof of \cite[Lemma~3.5.3]{CDL25} is non-extensible and it can be verified again directly that $\sfg_1+\ldots+\sfg_9$ is divisible by $2$.
\end{proof}

\begin{proposition}\label{prop_extendSequenceTwoWays}
Let $g_1,\ldots,g_k$ be an isotropic sequence of length $k<9$. Then, the sequence can be extended to an isotropic sequence $(g_1,\ldots,g_k,g_{k+1},\ldots,g_{10})$ of length $10$, and to an isotropic sequence $(g_1,\ldots,g_k,g'_{k+1},\ldots,g'_9)\in\One$.
\end{proposition}

\begin{proof}
The extension to a length $10$ sequence is \cite[Corollary~6.1.2]{DK25}. The second statement has an analogous proof: consider an isotropic sequence $(f_1,\ldots,f_9) \in \One$ and use Proposition~\ref{prop_transitivity} to find $w\in W(L_{10})$ such that $w(g_i)=f_i$ for $i=1,\ldots, k$. Then, set $g_{k+i}'=w^{-1}f_{k+i}$ for $i=1,\ldots,9-k$.
\end{proof} 

Let $W^\mathrm{nod}_Y$ denote the subgroup of $W(L_{10})$ generated by reflections with respect to classes of $(-2)$-curves on an Enriques surface $Y$. Then, $\wnod_Y$ acts on the set of isotropic sequences, and every orbit of this action contains a distinct representative, called a \textit{canonical} isotropic sequence. This is described explicitly in the following proposition.

\begin{proposition}[{\cite[Proposition~6.1.5]{DK25}}]
\label{lem:CanonicalSequences}
Let $Y$ be an Enriques surface and suppose that $(f_1,\ldots,f_k)$ is an isotropic sequence in $S_Y$.
Then, there exists a unique $w\in \wnod_Y$ such that, up to reordering:
\begin{itemize}

\item the sequence $(f_1',\ldots,f_k')\coloneqq (w(f_1),\ldots,w(f_k))$ contains a non-degenerate subsequence $(f_{i_1}',\ldots,f_{i_c}')$ with $1=i_1< \ldots <i_c$;

\item for any $i_s<i<i_{s+1}$ there are smooth rational curves $R^{i_s}_1,\ldots,R^{i_s}_{i-i_s}$ such that 
\[
f_i'=f_{i_s}' + R^{i_s}_1 + \ldots + R^{i_s}_{i-i_s} \in W^\mathrm{nod}_Y\cdot f_{i_s}.
\]
Here, $R^{i_s}_1 + \ldots + R^{i_s}_{i-i_s}$ is a chain of rational curves with dual graph $A_{i-i_s}$. Moreover, $f_{i_s}'$ is disjoint from $R_{2}^{i_s},\ldots,R^{i_s}_{i-i_s}$ and $f_{i_s}'\cdot R^{i_s}_{1}=1$.
\end{itemize}
The isotropic sequence $(f_1',\ldots,f_k')$ is said to be \emph{canonical} and a smooth rational curve $R_j^{i_s}$ as above is said to \emph{appear} in $(f_1',\ldots,f_k')$.
\end{proposition}

\begin{definition}\label{def_ndForSequence}
The number $c$ of nef classes in a canonical isotropic sequence $\bff$ is the \textit{non-degeneracy} of the sequence. In this case, we say that $\bff$ is $c$-non-degenerate.
\end{definition}

\begin{lemma}\label{lem_maximal sequences}
    Let $\bff \coloneqq(f_1,\ldots,f_m)$ be a canonical isotropic sequence with non-degeneracy $c$. Then either:
    \begin{enumerate}
        \item $m\leq 8$. In this case, both of the following hold:
        \begin{enumerate}
            \item $\bff$ can be extended to a canonical sequence $(f_1,\ldots,f_m,f'_{m+1},\ldots,f'_{10})$ of length $10$ and non-degeneracy $c'\geq c$, and
            \item $\bff$ can be extended to a canonical sequence $(f_1,\ldots,f_m,f''_{m+1},\ldots,f''_{9}) \in \One$ with non-degeneracy $c''\geq c$.
        \end{enumerate}
        \item $m=9$. Then either:
        \begin{enumerate}
            \item $\bff \in \Oext$ and there exists a canonical isotropic sequence $(f_1,\ldots, f_9,f'_{10})$ with non-degeneracy $c'\geq c$, or
            \item $\bff \in \One$.
        \end{enumerate}
        \item $m=10$, so $\bff$ has maximal length.
    \end{enumerate}
\end{lemma}

\begin{proof}
Suppose that $m\leq 8$. Then, (1)~(a) is the content of \cite[Proposition~6.1.7]{DK25}.
    The proof of (1)~(b) is analogous to the proof of \cite[Proposition~6.1.7]{DK25}, with the only difference that we use Proposition~\ref{prop_extendSequenceTwoWays} instead of \cite[Corollary~6.1.2]{DK25}. Suppose now $m=9$. Then, (2)~(a) is analogous to the one of \cite[Proposition~6.1.7]{DK25}, but the extension to an isotropic sequence of length $10$ is given by the hypothesis $\bff \in \Oext$. In the remaining cases (2)~(b) and (3) there is nothing to show.
\end{proof}

\subsection{Fano and Mukai polarizations}
\label{sec:Fano-Mukai-polarizations}

 Lemma~\ref{lem_maximal sequences} underlines the importance of two classes of isotropic sequences, namely canonical isotropic sequences of length $10$ and canonical isotropic sequences of length $9$ in $\One$. Indeed, any other canonical isotropic sequence can be extended to one of these types. These two classes of isotropic sequences carry information about the birational geometry of Enriques surfaces. We illustrate this principle, and start by recalling the definition of the $\Phi$ function.

 \begin{definition}
For an Enriques surface $Y$, let $\Phi\colon S_Y \to \bZ_{\geq 0}$ be the function defined as
\[
x\mapsto\min\{ | x \cdot f | \colon f \in S_Y \mbox{ is isotropic} \}.
\]
In the special case where $x=[D]$ for $D$ big divisor, then \cite[Lemma~2.4.10~(2)]{CDL25} implies that
\[
\Phi([D])=\min\{ | [D] \cdot f | \colon f \in S_Y \mbox{ is a half-fiber} \}.
\]
\end{definition}

\subsubsection{Fano polarizations}

\begin{definition}
    On an Enriques surface $Y$, a nef divisor $D$ satisfying $D^2=10$ and $\Phi(D)=3$ is called a \textit{Fano polarization}.
    The image $Y'\subseteq \pr 5$ of $Y$ through $\phi_{|D|}$ is called a \emph{Fano model}. 
    The numerical class of a Fano polarization is called a \textit{numerical Fano polarization}.
\end{definition}

The following lemma is well-known to experts:
it expresses the relation between Fano polarizations and isotropic sequences of length $10$. We include a proof for the reader's convenience.

\begin{lemma}
\label{lem_FanoPolCanonicalSeq}
For an Enriques surface $Y$, the following hold:
\begin{enumerate}

\item Let $h$ be a numerical Fano polarization on $Y$. Then, there exists a canonical isotropic sequence $(f_1,\ldots,f_{10})$ such that
\[
h=\frac{1}{3}(f_1 + \ldots + f_{10}).
\]

\item Conversely, let $\bff\coloneqq(f_1,\ldots,f_{10})$ be a canonical isotropic sequence. Then, $\frac{1}{3}(f_1 + \ldots + f_{10})$ is a numerical Fano polarization.

\end{enumerate}
\end{lemma}

\begin{proof}
Part~(1) is \cite[Lemma~3.5.2]{CDL25}. Conversely, given $\bff$, we define $h=\frac{1}{3}(f_1 + \ldots + f_{10})$. We have that $\Phi(h)\leq3$ as $h\cdot f_i=3$ for all $i=1,\ldots,10$. Now we show that $h$ is nef. It follows from the form of a canonical isotropic sequence that $h\cdot R\geq 0$ for all curves $R$ appearing in $\bff$. Thus, for an effective divisor $D$, we have $h\cdot[D]\geq 0$ because possible common components between $h$ and $[D]$ contribute non-negatively to their intersection. Hence, $h$ is nef, and also big because $h^2=10$. To prove that $\Phi(h)=3$, it suffices to show that $h\cdot v>3$ where $v$ is any class of a half-fiber different from $f_1,\ldots,f_{10}$. Because of this assumption, we have that $f_i\cdot v>0$, and hence $h\cdot v = \frac{1}{3} \sum_{i=1}^{10}f_i\cdot v > 3$.
\end{proof}

The decomposition in Lemma~\ref{lem_FanoPolCanonicalSeq}~(1) is unique. 

\begin{lemma}[{\cite[Lemma~2.7]{MMV_Reye}}]\label{lem_UniqueCanonicalFano}
Let $h=[D]$ be a numerical Fano polarization with associated canonical isotropic sequence $(f_1,\ldots,f_{10})$ as in Lemma~\ref{lem_FanoPolCanonicalSeq}~(1). Then,
\begin{enumerate}

\item If $E$ is an effective isotropic divisor, then $[E]$ coincides with one of the $f_i$ if and only if $D\cdot E=3$. In particular, the sequence $(f_1,\ldots,f_{10})$ is unique up to reordering;

\item If $R$ is a $(-2)$-curve, then $h \cdot R = 0$ if and only if $R$ appears in the sequence $(f_1,\ldots,f_{10})$.

\end{enumerate}
\end{lemma}

\subsubsection{Mukai polarizations}
As discussed above, Fano polarizations are closely related to canonical isotropic sequences of maximal length $10$. 
On the other hand, canonical isotropic sequences in $\One$ are related to another type of polarizations called Mukai polarizations \cite[Section~3.5]{CDL25}, which we now recall.

\begin{definition} \label{defn:MukaiPol}
On an Enriques surface $Y$, a \emph{Mukai polarization} $D$ is a nef divisor $D$ such that $D^2=18$ and $\Phi(D)=4$. The image $Y'\subseteq \pr 9$ of $Y$ through $\phi_{|D|}$ is called a \emph{Mukai model}. 
A \emph{numerical Mukai polarization} is the numerical equivalence class of a Mukai polarization.
\end{definition}

\begin{lemma}
\label{lem_MukaiPolCanonicalSeq}
For an Enriques surface $Y$, the following hold:
\begin{enumerate}

\item Let $v$ be a numerical Mukai polarization on $Y$. Then, there exists a canonical isotropic sequence $(g_1,\ldots,g_{9})\in \One$ such that
\[
v=\frac{1}{2}(g_1 + \ldots + g_{9}).
\]

\item Conversely, let $(g_1,\ldots,g_9)\in \One$ be a canonical isotropic sequence. Then, $\frac{1}{2}(g_1 + \ldots + g_9)$ is a numerical Mukai polarization.
\end{enumerate}
\end{lemma}

\begin{proof}
Part~(1) follows from the proof of \cite[Lemma~3.5.3]{CDL25} (more precisely, see \cite[Equation~(3.5.5)]{CDL25}). For Part~(2), consider a canonical isotropic sequence $\bfg\coloneqq (g_1,\ldots,g_9)$. Then, by Proposition~\ref{prop_OrbitsLength9}, $\bfg\in \One$ if and only if $v\coloneqq \frac{1}{2}(g_1 + \ldots + g_9)$ belongs to $S_Y$. In this case, we have that $v^2=18$. Moreover by the proof of \cite[Lemma~3.5.3]{CDL25}, we have that $\Phi(v)=4$. So, to show that $v$ is a numerical Mukai polarization, it suffices to see that $v$ is nef.

Suppose for the sake of contradiction that there exists a $(-2)$-curve $R$ such that $v\cdot R < 0$. Then, there exists $i\in\{1,\ldots,9\}$ such that $g_i \cdot R <0$. Since $\bfg$ is canonical, this is only possible if $g_{i}=g_{i-1} + R$ for some $i\in\{2,\ldots,9\}$ with $g_{i-1}\cdot R=1$, $g_{i}\cdot R=-1$, and $g_j\cdot R=0$ for all $j\neq i-1,i$. Then, we must have that
\begin{equation}
    \label{eq_vR0}
2v\cdot R=g_{i-1}\cdot R+g_{i}\cdot R+\sum_{j\neq i-1,i} g_j\cdot R=0,
\end{equation}
which is a contradiction as it implies that $v\cdot R=0$.
\end{proof}

\begin{lemma}\label{lem_UniqueCanonicalMukai}
Let $v=[D]$ be a numerical Mukai polarization with associated canonical isotropic sequence $(g_1,\ldots,g_9)$ as in Lemma~\ref{lem_MukaiPolCanonicalSeq}~(1). Then,
\begin{enumerate}
    \item If $E$ is an effective isotropic divisor, then $[E]$ coincides with one of the $g_i$ if and only if $D\cdot E=4$. In particular, the sequence $(g_1,\ldots,g_9)$ is unique up to reordering;
    \item If $R$ is a $(-2)$-curve, then $v \cdot R = 0$ if and only if $R$ appears in $(g_1,\ldots,g_9)$, or $g_i\cdot R = 0$ for all $i=1,\ldots,9$.
\end{enumerate}
\end{lemma}

\begin{proof}
Part (1) is analogous to the proof of \cite[Lemma~2.7~(2)]{MMV_Reye}, after appropriately adjusting the intersection numbers.

For part (2), if $R$ appears in $(g_1,\ldots, g_9)$ then $v\cdot R = 0$ as in \eqref{eq_vR0}. If $R$ satisfies $g_i \cdot R = 0$ for $i=1,\ldots,9$ then $v \cdot R =0$.
Conversely, if $v\cdot R = 0$ then either $g_i \cdot R = 0$ for all $i$, or there exists $i$ such that $g_i\cdot R < 0$, which is only possible if $R$ appears in $(g_1,\ldots,g_9)$.
\end{proof}

\begin{remark}
The vectors of a canonical isotropic sequence of length $10$ form a basis for $S_Y\otimes \bQ$, so no $(-2)$-curve is orthogonal to all of them. This no longer holds for an isotropic sequence of lenght $9$, which explains the additional condition in Lemma~\ref{lem_UniqueCanonicalMukai} (2). \end{remark}

\begin{lemma}\label{lem_uniqueOrt}
    Let $v=\frac{1}{2}(g_1 + \ldots +g_9)$ be a numerical Mukai polarization on an Enriques surface $Y$. 
    The orthogonal complement $H:=\langle g_1,\ldots,g_9\rangle^\perp\subseteq S_Y$ is a lattice isometric to $A_1$. At most one $(-2)$-vector in $H$ is effective, in which case it is a $(-2)$-curve.
\end{lemma}

\begin{proof}
We first prove that $H$ is isometric to $A_1$. We have that $H$ has rank $1$ because $g_1,\ldots,g_9$ are independent. Next, we exhibit an explicit $(-2)$-vector in $H$.
By Proposition~\ref{prop_OrbitsLength9} there exists $w\in W(L_{10})$ such that $w(g_i)=\sfg_i$, where $(\sfg_1,\ldots,\sfg_9)$ is the non-extensible isotropic sequence in the proof of \cite[Lemma~3.5.3]{CDL25}. More precisely, we can write $\sfg_i = \Updelta - \sff_i - \sff_{10}$ for $i\in\{1,\ldots,9\}$ (see \cite[Proposition~1.5.3]{CDL25}). Now, we have that $\Updelta - 2\sff_{10}$ is a $(-2)$-vector orthogonal to $\sfg_1,\ldots,\sfg_9$. Hence,
\[
\rho:=w^{-1}(\Updelta - 2\sff_{10})
\]
is a $(-2)$-vector in $H$. We have that $\rho$ is necessarily primitive, and hence it is a generator for $H$, showing the claimed isometry with $A_1$. Additionally, since $H$ has rank $1$, we have that $\pm\rho$ are the only $(-2)$-vectors in $H$. As $\pm\rho$ are opposite, at most one of them is effective.

Finally, we show that if one among $\pm\rho$ is effective, then it has to be the class of a $(-2)$-curve. Without loss of generality, say that $\rho$ is effective. Because $v\cdot \rho = 0 $, by \cite[Proposition~2.1.6~(2)]{CDL25} we may write $\rho$ as a positive linear combination $\sum_k a_k R_k$ of $(-2)$-curves. We have $v \cdot R_k = 0$ for all $k$, because $a_k>0$ and $v$ is nef with $v\cdot \rho =0$. Then, either $R_k \in H$ for some $k$, in which case we must have $R_k = \rho$ (the unique effective $(-2)$-vector in $H$), or all the $R_k$ appear in $(g_1,\ldots, g_9)$. In this case, $\rho \in H$ implies that $\rho \cdot R_k = 0$ for all $k$, whence
\[
-2=\rho^2=\rho\cdot\sum_ka_kR_k=\sum_ka_k(\rho\cdot R_k)=0,
\]
which is a contradiction.
\end{proof}

\begin{remark}
With reference to Lemma~\ref{lem_uniqueOrt}, the general Enriques surface does not have smooth rational curves, therefore the two $(-2)$-vectors in $H$ cannot be effective. The other possibility can also happen for special Enriques surfaces: in Example~\ref{exa_nefMukai} we provide a numerical Mukai polarization on an Enriques surface with a $(-2)$-curve in $H$.
\end{remark}

\subsection{A refinement of the non-degeneracy invariant}

Motivated by the relation between canonical isotropic sequences of length $10$ and canonical sequences in $\One$ with the birational geometry of Fano and Mukai models, we give the following definitions. Let $Y$ be an Enriques surface. 

\begin{definition}\label{def_Fnd}
The \emph{Fano non-degeneracy invariant} of $Y$, $\Fnd(Y)$, is the maximum integer $c$ for which there exists a canonical isotropic sequence of length $10$ with non-degeneracy $c$.
\end{definition}

\begin{definition}\label{def_Mnd}
The \emph{Mukai non-degeneracy invariant} of $Y$, $\Mnd(Y)$, is the maximum integer $c$ for which there exists a canonical isotropic sequence in $\One$ with non-degeneracy $c$.
\end{definition}

\begin{lemma}\label{lem_NdGeqMax}
    We have that $\nd(Y)\geq \Fnd(Y)$ and $\nd(Y)\geq \Mnd(Y)$.
\end{lemma}

\begin{proof}
Given any length $10$ canonical isotropic sequence, its non-degeneracy $c$ satisfies $\nd(Y)\geq c$. Hence, $\nd(Y)\geq\Fnd(Y)$. Similarly, by considering a canonical isotropic sequence in $\One$, its non degeneracy $c$ also satisfies $\nd(Y)\geq c$, from which follows that $\nd(Y)\geq\Mnd(Y)$.
\end{proof}

\begin{proposition} \label{prop:casesNdFndMnd}
We have $\nd(Y)=\Fnd(Y)=\Mnd(Y)$, except possibly when:
\begin{enumerate}[(i)]
    \item $\nd(Y)=10=\Fnd(Y)$ and $\Mnd(Y)= 9$;
    \item $\nd(Y)=10=\Fnd(Y)$ and $\Mnd(Y)= 8$;
    \item $\nd(Y)=9=\Fnd(Y)$, in which case $\Mnd(Y)=8$;
    \item $\nd(Y)=9=\Mnd(Y)$, in which case $\Fnd(Y)=8$.
\end{enumerate}
\end{proposition}

\begin{proof}
Let $\bff = (f_1,\ldots,f_{\nd(Y)})$ be a non-degenerate isotropic sequence. If $\nd(Y) \leq 8$,  then by Lemma~\ref{lem_maximal sequences}~(1)~(a), $\bff$ extends to a sequence of length $10$ of non-degeneracy $c'\geq \nd(Y)$. Thus, $\Fnd(Y)\geq\nd(Y)$, and hence $\Fnd(Y)= \nd(Y)$ by Lemma~\ref{lem_NdGeqMax}. In the same way, applying Lemma~\ref{lem_maximal sequences}~(1)~(b) to $\bff$ we obtain that $\nd(Y)=\Mnd(Y)$.

Assume now that $\nd(Y)= 9$. If $\bff \in \Oext$, then by Lemma~\ref{lem_maximal sequences}~(2) applied to $(f_1,\ldots,f_9)$ we have that $\Fnd(Y) \geq 9$, and hence $\Fnd(Y) = \nd(Y) = 9$ by Lemma~\ref{lem_NdGeqMax}. Then, if $\Mnd(Y)=8$ we recover case (iii), otherwise the three invariants coincide again. On the other hand, if $\bff \in \One$ then $\Mnd(Y) = 9$. Thus, the three invariants coincide, or we are in case (iv).

Finally, if $\nd(Y) = 10$, then by definition also $\Fnd(Y)=10$. Applying Lemma~\ref{lem_maximal sequences}~(1)~(b) to $(f_1,\ldots,f_8)$ shows that $\Mnd(Y)\geq 8$, which yields the remaining cases: (i) and (ii).
\end{proof}

\begin{corollary}\label{cor:max}
We have that $\nd(Y)=\max\{\Fnd(Y),\Mnd(Y)\}$.
\end{corollary}

\begin{proof}
This is true because the claimed equality holds in all the different cases in Proposition~\ref{prop:casesNdFndMnd}.
\end{proof}

\begin{remark}
A general Enriques surface $Y$ does not contain smooth rational curves. So, in this case, $\nd(Y)=\Fnd(Y)=10$ and $\Mnd(Y)=9$. If $\nd(Y)\leq8$, then the three non-degeneracy invariants coincide, and examples of Enriques surfaces such that $\nd(Y)\leq8$ are known. On the other hand, we do not know examples of Enriques surfaces satisfying (ii), (iii), or (iv) from Proposition~\ref{prop:casesNdFndMnd}. See Section~\ref{sec:Mukai} for further observations and questions.
\end{remark}

\subsection{Upper bounds for Fano and Mukai non-degeneracy invariants}
In Theorem~\ref{cor_ndUpperBound} we produce an upper bound for the Fano and Mukai non-degeneracy invariants of an Enriques surface $Y$. 
Let $\mathrm{aut}(Y)$ denote the subgroup of isometries of $S_Y$ induced by the action of the automorphism group $\mathrm{Aut}(Y)$. The next objects we introduce will be often used in the sequel.

\begin{definition}
\label{def:set-of-polarizations}
For an Enriques surface $Y$, denote by $\sfF_Y \subseteq S_Y$ (resp. $\sfM_Y \subseteq S_Y$) the set of all its numerical Fano polarizations (resp. its numerical Mukai polarizations). A \emph{set of representatives for the orbits of the action of $\aut(Y)$ on $\sfF_Y$} will be a subset $\sfF_Y'\subseteq\sfF_Y$ such that $\sfF_Y'/\aut(Y)=\sfF_Y/\aut(Y)$. We remark that we allow for different elements in $\sfF_Y'$ to possibly be in the same $\aut(Y)$-orbit. We define in an analogous way a set of representatives $\sfM_Y'$ for the orbits of the action of $\aut(Y)$ on $\sfM_Y$.
\end{definition}

We have then the following upper bounds on the Fano and Mukai non-degeneracy invariants.

\begin{theorem}
\label{cor_ndUpperBound}
Let $Y$ be an Enriques surface and let $\sfF_Y'$ and $\sfM_Y'$ as in Definition~\ref{def:set-of-polarizations}.
\begin{enumerate}
    \item The following two conditions are equivalent:
    \begin{itemize}
        \item $\Fnd(Y)\leq10-d$ for some positive integer $d$;
        \item for every $h\in\sfF_Y'$ there exist at least $d$ distinct smooth rational curves $R_1,\dots,R_d \subseteq Y$ such that $h\cdot R_i=0$ for $i=1,\ldots,d$.
    \end{itemize}
\item The following two conditions are equivalent:
    \begin{itemize}
    \item $\Mnd(Y)\leq 9-d$ for some positive integer $d$;
    \item for every $v\in\sfM_Y'$, with associated canonical isotropic sequence $(g_1,\ldots,g_9)$ and $H = \langle g_1,\ldots,g_9 \rangle^\perp$, there exist at least $d$ distinct smooth rational curves $R_1,\dots,R_d \subseteq Y$ such that $v\cdot R_i=0$ and $R_i \notin H$ for $i=1,\ldots,d$.
    \end{itemize}
\end{enumerate}
This is independent from the choice of orbit representatives in $\sfF_Y'$ and $\sfM_Y'$.
\end{theorem}

\begin{proof}
Suppose that, for any element in $\mathsf{F}_Y'$, the condition in the statement holds. Then, to prove that $\Fnd(Y)\leq10-d$, it suffices to show that every canonical isotropic sequence $\mathbf{f}\coloneqq (f_1,\ldots, f_{10})$, has non-degeneracy $c\leq 10-d$.

Fix a canonical sequence $\bff$, and let $h$ be the associated Fano polarization (as in Lemma~\ref{lem_FanoPolCanonicalSeq}~(2)). 
By definition of $\mathsf{F}_Y'$, there exists $\phi\in\mathrm{aut}(Y)$ such that $\phi(h)\in\mathsf{F}_Y'$. So, by hypothesis, there exist distinct smooth rational curves $R_1,\ldots,R_d$ such that $\phi(h)\cdot R_i = 0$ for $i=1,\ldots,d$. Since $\phi$ is induced on $S_Y$ by an automorphism of $Y$, we have that $S_i\coloneqq \phi^{-1}(R_i)$ for $i=1,\ldots,d$ still correspond to classes of distinct smooth rational curves on $Y$ which satisfy $h\cdot S_i = \phi(h)\cdot R_i = 0$. 
By Lemma~\ref{lem_UniqueCanonicalFano}~(2), the curves $S_1,\ldots,S_d$ appear in the sequence $\bff$. Again by the definition of a canonical sequence (Proposition~\ref{lem:CanonicalSequences}), this implies that there are $d$ distinct non-nef classes in $\bff$. Thus, by definition, the sequence $(f_1,\ldots,f_{10})$ is $c$-non-degenerate with $c\leq10-d$. 

Conversely, assume that $\Fnd(Y)\leq10-d$. Pick $h\in\sfF_Y'$. By Lemma~\ref{lem_FanoPolCanonicalSeq}~(1), there exists a canonical isotropic sequence $\bff$ such that $h=\frac 13(f_1+\ldots+f_{10})$. By assumption, $\bff$ is $c$-non-degenerate with $c\leq 10-d$. By the structure of a canonical sequence (Proposition~\ref{lem:CanonicalSequences}) and Definition~\ref{def_ndForSequence}, since $\bff$ is $c$-non-degenerate then there are $10-c$ distinct smooth rational curves $R_i$ appearing in $\bff$. We have that $h\cdot R_i=0$ for all $i$ by Lemma~\ref{lem_UniqueCanonicalFano} (2). For $d \leq 10-c$ (so that $c\leq 10- d$), the collection $\{R_1,\ldots,R_d\}$ is a collection of distinct smooth rational curves satisfying $h \cdot R_i=0$. 

The proof in the case of Mukai polarizations is the same, using Lemma~\ref{lem_MukaiPolCanonicalSeq} and Lemma~\ref{lem_UniqueCanonicalMukai} instead of Lemma~\ref{lem_FanoPolCanonicalSeq} and Lemma~\ref{lem_UniqueCanonicalFano}.
\end{proof}

\subsection{Enriques surfaces with finite automorphism groups: numerical Fano and Mukai polarizations}
\label{sec:num-Fano-pol-En-surf-with-finite-aut}

Let $Z$ be an Enriques surface with finite automorphism group. Then, $Z$ has finitely many smooth rational curves and elliptic fibrations. Any canonical isotropic sequence in $Z$ is constructed by combining elements of these two sets, and therefore there are only finitely many possibilities for a Fano or Mukai polarization, which is determined uniquely by its canonical sequence as in Lemma~\ref{lem_FanoPolCanonicalSeq} and Lemma~\ref{lem_MukaiPolCanonicalSeq}. In other words, $\sfF_Z$ and $\sfM_Z$ are finite sets. We obtain a complete list of such polarizations by an exhaustive computer search. This list, which will be used in the proof of Theorem~\ref{thm:ndTauBarTau}, is available at \cite{DataNew}. In \eqref{table:|FZ|-|MZ|-Kondo} we report the cardinalities of $\sfF_Z,\sfM_Z$ and the number of orbits modulo $\mathrm{aut}(Z)$ (we include this information for completeness, although it is not needed for the rest of the article).

\begin{equation}\label{table:|FZ|-|MZ|-Kondo}
\begin{tabular}{c|c|c|c|c|c|c|c}
Type of $Z$ in \cite{Kon86}              & I     & II  & III  & IV    & V     & VI   & VII  \\ \hline
$|\sfF_Z|$         &  74   & 573 & 2712 & 15336 & 1672  & 9031 & 8561 \\ 
$|\sfF_Z/\aut(Z)|$ &  20   & 32  & 61   & 68    & 79    & 87   & 81  \\
\hline
$|\sfM_Z|$         & 142   & 1227& 6944 &  43038     & 4188  & 24665 & 22425   \\ 
$|\sfM_Z/\aut(Z)|$ & 38    & 64  & 147  &  179     & 323   & 872   & 1034   \\
\end{tabular}
\end{equation}

\begin{remark}
The computer search for elements of $\sfF_Z$ and $\sfM_Z$ starts from the list of saturated non-degenerate isotropic sequences in \cite[Section~7]{MRS22Paper}. We take all the possible subsequences and then extend them by using the finitely many smooth rational curves. This produces all the isotropic sequences of length $10$ and the isotropic sequences of length $9$ in $\One$. Hence, we obtain $\sfF_Z$ and $\sfM_Z$.
\end{remark}

\section{Specialization of the Non-degeneracy invariant in families}\label{sec3}

In this section we prove a specialization result for the non-degeneracy invariants $\nd$, $\Fnd$, and $\Mnd$ for Enriques surfaces in families.

\begin{theorem} \label{thm:family}
Let $Y\to T$ be a smooth projective morphism of schemes over $\bC$ of relative dimension $2$. Suppose that $T$ is irreducible and Noetherian and that for a closed point $0\in T(\bC)$, the fiber $Y_0$ satisfies $H^1(Y_0, \sO_{Y_0}) = H^2(Y_0, \sO_{Y_0}) = 0$. Then, for very general $t\in T(\bC)$ there is an identification $\Pic(Y_0)\cong \Pic(Y_t)$ which preserves intersections and induces an inclusion $\Nef(Y_0) \subseteq \Nef(Y_t)$.
\end{theorem}

\begin{proof}
Let $\overline{\eta}$ be the geometric generic point of $T$: that is, if $\eta$ is the generic point of $T$, then $\overline{\eta}$ denotes the morphism $\Spec\left(\overline{k(\eta)}\right)\rightarrow T$, where $k(\eta)$ is the residue field of $\eta$. By \cite[Lemma~2.1]{Vial13}, the scheme $Y_{\overline{\eta}}$ is abstractly isomorphic to $Y_t$, where $t \in T(\bC)$ is very general. Then, it suffices to show that there is an identification $\Pic(Y_{\overline{\eta}})\cong \Pic(Y_0)$ which preserves the intersections and induces an inclusion $\Nef(Y_0) \subseteq \Nef(Y_{\overline{\eta}})$.

We then make the following standard reduction. Let $K$ be the residue field $k(\eta)$. By \cite[\href{https://stacks.math.columbia.edu/tag/0CM2}{Tag 0CM2}]{stacks-project}, there exists a discrete valuation ring $R$ with fraction field $K$ and a morphism $\Spec(R) \to T$ such that the closed point of $\Spec(R)$ maps to $0\in T$ and the restriction $\Spec(K)\to T$ has image $\eta$. Consider the pulled back family $Y_R=Y\times_T \Spec(R) \to \Spec(R)$. The geometric generic fiber of $Y_R\rightarrow\Spec(R)$, which is $Y_{\overline{K}}$, is isomorphic to $Y_{\overline{\eta}}$. Moreover, the fiber over the closed point of $Y_R\rightarrow\Spec(R)$ is isomorphic to $Y_0$. Thus, we can assume that $T=\Spec(R)$ with $0\in T$ the unique closed point.

Next, replacing $R$ with its henselianization $R\hookrightarrow H$, replaces the residue field $k\cong\mathbb{C}$ and the field of fractions $K$ of $R$ by appropriate algebraic extensions $k\hookrightarrow k'$ and $K\hookrightarrow K'$ (for basic facts about henselian rings we point the reader to \cite[Section~2.3]{NeronModels}).
We now identify the special fiber and the geometric generic fiber of the pulled back family $Y_H\rightarrow\Spec(H)$. The special fiber is still $Y_0$ because $k'\cong k$. For the geometric generic fiber, let us choose an algebraic closure $\overline{K}$ which contains $K'$. Hence, $\overline{K}$ is also an algebraic closure of $K'$. Thus, $Y_H \to \Spec(H)$ has the same geometric generic fiber as $Y\to \Spec(R)$. Therefore, we may replace $R$ with its henselization and assume that $R$ is a henselian discrete valuation ring.

With this henselian assumption, we fall within the same set of assumptions as in the proof of \cite[Theorem~4.1]{TotaroNef}, and by the same arguments we have that the restriction maps $\Pic(Y) \to \Pic(Y_K)$, $\Pic(Y)\to \Pic(Y_0)$, and the pullback map $\Pic(Y_K) \to \Pic(Y_{\overline{K}})$ are isomorphisms (to show this, \cite{TotaroNef} uses the assumption that $H^1(Y_0, \sO_{Y_0}) = H^2(Y_0, \sO_{Y_0}) = 0$).

The composition $\Pic(Y_0) \cong \Pic(Y_{\overline{K}})$ preserves the intersection product of divisors: indeed, the intersection is determined by the Euler characteristic of line bundles on $Y$ \cite[Definition~1.7]{Debarre_book}, which are preserved by flatness of the family \cite[Chapter~III, Section~9.9]{Har77}. Hence, the isomorphism induces an isomorphism of $\mathbb{R}$-vector spaces $N^1(Y_0)_{\mathbb{R}}\cong N^1(Y_{\overline{K}})_{\mathbb{R}}$.

Next, we show that $\Nef(Y_0)\subseteq\Nef(Y_{\overline{K}})$ through the above isomorphisms. Let $L_0$ be a nef line bundle on $Y_0$ and denote by $L_{\overline{K}}$ the pullback to $Y_{\overline{K}}$ of an extension of $L_0$ to $Y$. We want to prove that $L_{\overline{K}}$ is nef. For this purpose, let $C_{\overline{K}}$ be an effective curve in $Y_{\overline{K}}$ and let us verify that $L_{\overline{K}}\cdot C_{\overline{K}}\geq0$.

Since the pullback map $\Pic(Y_K) \to \Pic(Y_{\overline{K}})$ is an isomorphism, we have that $C_{\overline{K}}$ is the pullback of a divisor $C_K$, unique up to linear equivalence. By \cite[Theorem~14.85]{GW10} we have that also $C_K$ is effective.
Let $C$ be the image of $C_K$ via the inverse of the restriction map $\Pic(Y) \to \Pic(Y_K)$. We have that $C$ is effective on the generic fiber of $Y\rightarrow\Spec(R)$ because, by definition, $C$ restricts to it giving $C_K$. Therefore, by semicontinuity \cite[Chapter~III, Section~12.8]{Har77}, we have that $C_0:=C|_{Y_0}$ is effective. Since the intersection products are preserved, we have that
\[
L_{\overline{K}}\cdot C_{\overline{K}} = L_0 \cdot C_0,
\]
where $L_0 \cdot C_0 \geq 0$ by nefness of $L_0$ and effectivity of $C_0$. This implies what we needed.
\end{proof}

\begin{corollary}\label{cor:family}
Let $Y\to T$ be a family of Enriques surfaces with $T$ irreducible and Noetherian scheme. Let $0\in T(\bC)$ be a closed point. Then, $\nd(Y_0)\leq \nd(Y_{t})$ for a very general $t \in T(\bC)$. Moreover, also $\Fnd(Y_0)\leq \Fnd(Y_{t})$, $\Mnd(Y_0)\leq \Mnd(Y_{t})$.
\end{corollary}

\begin{proof}
The family $Y \to T$ satisfies the hypotheses of Theorem \ref{thm:family}, so we have an inclusion $\Nef(Y_0) \subseteq \Nef(Y_t)$. Hence, every non-degenerate isotropic sequence of $Y_0$ can be viewed as a non-degenerate isotropic sequence of $Y_t$ of the same length. From this follows that $\nd(Y_{0})\leq\nd(Y_t)$.

Next, we consider the case of the Fano non-degeneracy invariant. Let $(f_1,\ldots,f_{10})$ be a isotropic sequence on $Y_0$ of non-degeneracy $c$. Via the isomorphism $\Pic(Y_0)\cong\Pic(Y_t)$ and the containment $\Nef(Y_0) \subseteq \Nef(Y_t)$, we have that $(f_1,\ldots,f_{10})$ corresponds to a length $10$ isotropic sequence in $Y_t$ also of non-degeneracy $c$. Hence, $\Fnd(Y_{0})\leq\Fnd(Y_t)$.

Lastly, for the Mukai non-degeneracy invariant, consider a non-extensible isotropic sequence $(f_1,\ldots,f_9)$ on $Y_0$. Again, the above argument applies if we can show that the corresponding isotropic sequence in $Y_t$ of length $9$ is also non extensible. However, this property is characterized by the intersection product (see Proposition~\ref{prop_OrbitsLength9}), which is preserved by the isomorphism $\Pic(Y_0)\cong\Pic(Y_t)$ as shown in the proof of Theorem~\ref{thm:family}.
\end{proof}

\section{Specialization of \texorpdfstring{$\ttbar$}{tau,bartau}-generic Enriques surfaces}\label{sec4}

In this section we develop an application of Corollary~\ref{cor:family} to certain families of Enriques surfaces. To construct these families we combine the general machinery for polarized K3 surfaces by Dolgachev \cite{Dol96} and Alexeev--Engel \cite{AE25} with the notion of $\ttbar$-generic Enriques surfaces introduced by Brandhorst--Shimada \cite{BS22} (see Definition~\ref{def:tautaubar-generic-Enriques-surfaces}).

The general result is Theorem~\ref{thm:boundLatticePol}. As a corollary, we compute the non-degeneracy invariant of the Enriques surfaces in the Barth--Peters family \cite{BP83} and in the first of the two families with eight disjoint smooth rational curves \cite[Example~1]{MLP02} (see Remark~\ref{rmk:Enriques-surfaces-MLP2-nd} for the second family).

\subsection{Lattice-polarized K3 surfaces}

Our exposition follows \cite{AE25}. Recall that a \emph{K3 surface} is a smooth connected projective complex surface $X$ with $\omega_X\cong\mathcal{O}_X$ and $h^1(X,\omega_X)=0$. For a K3 surface $X$ the lattice $H^2(X,\mathbb{Z})$ is independent of $X$ and is isometric to $L_{\mathrm{K3}}:=U^{\oplus3}\oplus E_8^{\oplus2}$, which is usually referred to as the \emph{K3 lattice}. An \emph{ADE K3 surface} $\overline{X}$ is a surface with only ADE singularities whose minimal resolution $X\rightarrow\overline{X}$ is a K3 surface.

Let $\Lambda$ be a primitive non-degenerate hyperbolic sublattice of the K3 lattice of singnature $(1, r - 1)$ with $r \leq 20$. Let $h\in \Lambda \otimes \bR$ be a vector contained in a small cone \cite[Definition~4.9]{AE25} satisfying $h^2>0$. A \emph{$(\Lambda,h)$-polarized K3 surface} is a pair $(\overline{X},j)$, where $\overline{X}$ is an ADE K3 surface, and $j$ is an isometric embedding $j \colon \Lambda \hookrightarrow \Pic(\overline{X})$ for which $j(h)$ is ample. A family of $(\Lambda,h)$-polarized K3 surfaces is a flat family $f\colon \overline{\sX}\to S$ of ADE K3 surfaces, together with a homomorphism $\Lambda \to \Pic_{\overline{\sX}/S}(S)$, such that every fiber is a $(\Lambda,h)$-polarized K3 surface. 

If instead we let $h$ lie in a generalized small cone \cite[Definition~5.5]{AE25} (this is always the case), let $h'$ be a positive norm vector in a small cone whose closure contains $h$. In this case, a family of \emph{$(\Lambda,h)$-polarized K3 surfaces} is a flat family $f\colon \overline{\sX}\to S$ of ADE K3 surfaces, together with a simultaneous partial resolution $\sX \to S$, which is a family of $(\Lambda,h')$-polarized K3 surfaces.

The moduli stack of $(\Lambda,h)$-polarized K3 surfaces is a smooth Deligne--Mumford stack which does not depend on $h$ \cite[Theorem~5.4 and Corollary~5.9]{AE25}. We denote it by $\sF_\Lambda$. Moreover, the universal family $\overline{\sX}\to \sF_{\lambda}$ admits a simultaneous resolution $\sX \to \sF_{\Lambda}$ to a family of $(\Lambda,h')$-polarized K3 surfaces \cite[Lemma~5.8]{AE25}. The fibers of $\sX$ are $(\Lambda,h')$-polarized K3 surfaces. 

For our purposes, we will make use of the following definition.

\begin{definition}
\label{def:very-general-polarized-K3}
We say that a K3 surface $X$ is a \textit{very general} $(\Lambda,h)$-polarized K3 surface if it is isomorphic to the geometric fiber of $\sX$ over the generic point of an irreducible component of $\sF_\Lambda$.
\end{definition}

\subsection{\texorpdfstring{$\ttbar$}{ttbar}-generic Enriques surfaces}
\label{ssec:ttbarDef}

An \emph{ADE-lattice} is an even, negative definite lattice $R$ generated by \emph{roots}, i.e. vectors $v\in R$ such that $v^2=-2$. An ADE-lattice has a basis consisting of roots whose associated dual graph is the disjoint union of some of the Dynkin diagrams $A_n$ ($n\geq1$), $D_n$ ($n\geq4$), and $E_6,E_7,E_8$. This \emph{ADE-type} for the lattice $R$ is denoted by $\tau(R)$.

In \cite{Shi21}, Shimada studied the ADE-lattices in the context of Enriques surfaces. More precisely, let $O^{\mathcal{P}}(L_{10})$ be the group of isometries of $L_{10}$ which preserve a positive half-cone $\mathcal{P}$, that is one of the two connected components of $\{v\in L_{10}\mid v\cdot v>0\}$. The following hold:
\begin{enumerate}

\item If $R_1,R_2\subseteq L_{10}$ are two ADE-lattices and $\overline{R}_1,\overline{R}_2$ denote their respective primitive closures in $E_{10}$, then also $\overline{R}_1,\overline{R}_2$ are ADE-lattices and $(\tau(R_1),\tau(\overline{R}_1))=(\tau(R_2),\tau(\overline{R}_2))$ if and only if $R_1$ and $R_2$ are in the same $O^{\mathcal{P}}(L_{10})$-orbit.

\item Let $R\subseteq L_{10}$ be an ADE-sublattice. There are $184$ possibilities for the pairs $(\tau(R),\tau(\overline{R}))$, which are listed in \cite[Table~1]{Shi21} (see also \cite[Table~1]{BS22}).
\end{enumerate}

\begin{definition}
\label{def:tautaubar-generic-Enriques-surfaces}
Let $Y$ be an Enriques surface with universal K3 cover $X\rightarrow Y$. Let $(\tau,\overline{\tau})$ one of the pairs in \cite[Table~1]{BS22}. Then $Y$ is called \emph{$(\tau,\overline{\tau})$-generic} provided the following hold:
\begin{enumerate}

\item Consider $H^{2,0}(X)\subseteq T_X\otimes\mathbb{C}$, where $T_X$ denotes the transcendental lattice of the K3 surface $X$. Then the group of isometries of $T_X$ preserving $H^{2,0}(X)$ is equal to $\{\pm\mathrm{id}_{T_X}\}$.

\item Let $R\subseteq L_{10}$ be an ADE-sublattice such that $(\tau(R),\tau(\overline{R}))=(\tau,\overline{\tau})$. Define $M_R$ to be the sublattice of $L_{10}(2)\oplus R(2)$ given by
\[
\langle(v,0),(w,\pm w)/2\mid v\in L_{10},~w\in R\rangle.
\]
Then, there exist isometries $M_R\cong S_X$ and $L_{10}\cong S_Y$ such that the following diagram commutes:

\begin{center}
\begin{tikzpicture}
\matrix(a)[matrix of math nodes,
row sep=2em, column sep=2em,
text height=1.5ex, text depth=0.25ex]
{L_{10}(2)&M_R\\
S_Y(2)&S_X.\\};
\path[right hook->] (a-1-1) edge node[]{}(a-1-2);
\path[->] (a-1-1) edge node[left]{$\cong$}(a-2-1);
\path[right hook->] (a-2-1) edge node[]{}(a-2-2);
\path[->] (a-1-2) edge node[right]{$\cong$}(a-2-2);
\end{tikzpicture}
\end{center}
\end{enumerate}
\end{definition}

Of the $184$ lattice-theoretic possibilities for the pairs $(\tau,\overline{\tau})$ classified in \cite{Shi21}, $155$ are realized by the $\ttbar$-generic Enriques surfaces. We will recall and use properties of these surfaces in Section \ref{sec:ttbarND}, where we compute values of (Fano and Mukai) non-degeneracy for them.

\subsection{Lattice polarized K3 covers of \texorpdfstring{$\ttbar$}{ttbar}-generic Enriques surfaces}

We begin by relaxing the notion of a $\ttbar$-generic Enriques surface.

\begin{definition}
\label{def:ttbar_surface}
Let $Y$ be an Enriques surface with universal K3 cover $\pi\colon X\to Y$. Let $(\tau,\overline{\tau})$ one of the pairs in \cite[Table~1]{BS22}. 
Let $R\subseteq L_{10}$ be an ADE-sublattice such that $(\tau(R),\tau(\overline{R}))=(\tau,\overline{\tau})$. Define $M_R$ as in Definition \ref{def:tautaubar-generic-Enriques-surfaces}.
We say that $Y$ is a \emph{$(\tau,\overline{\tau})$ Enriques surface} if there exists an isometry $L_{10} \cong S_Y$ and an isometric embedding $M_R \hookrightarrow S_X$ such that the following diagram commutes:

\begin{center}
\begin{tikzpicture}
\matrix(a)[matrix of math nodes,
row sep=2em, column sep=2em,
text height=1.5ex, text depth=0.25ex]
{L_{10}(2)&M_R\\
S_Y(2)&S_X.\\};
\path[right hook->] (a-1-1) edge node[]{}(a-1-2);
\path[->] (a-1-1) edge node[left]{$\cong$}(a-2-1);
\path[right hook->] (a-2-1) edge node[]{}(a-2-2);
\path[right hook->] (a-1-2) edge node[]{}(a-2-2);
\end{tikzpicture}
\end{center}

\end{definition}

If $Y$ is a $\ttbar$ Enriques surface, then its universal K3 cover $X$ is a smooth $(M_R,h)$-polarized K3 surface, where $h$ is the pullback of any ample class on $Y$. The next proposition gives a partial converse to this.

Fix $\ttbar$ to be one of the pairs in \cite[Table~1]{BS22}. Let $M_R$ be as in Definition \ref{def:tautaubar-generic-Enriques-surfaces}. Denote by $\alpha$ the inclusion of $L_{10}(2)$ into $M_R$ in Definition~\ref{def:tautaubar-generic-Enriques-surfaces}.

\begin{lemma}\label{lem:polK3}
Let $h\in\alpha(L_{10}(2))$ such that $h^2>0$. Let $X$ be a smooth $(M_R,h)$-polarized K3 surface.
Then, $X$ admits a unique Enriques involution $i$ whose action on $H^2(X,\bZ)$ is the identity on $\alpha(L_{10}(2))$. Moreover:
\begin{enumerate}
    \item The quotient $X/i$ is a $\ttbar$ Enriques surface.
\item If $X$ is a very general $(M_R,h)$-polarized K3 surface, then $X/i$ is $\ttbar$-generic.
\end{enumerate}
\end{lemma}

\begin{proof}   
Since $X$ is $(M_R,h)$-polarized and $L_{10}(2)$ is a primitive sublattice of $M_R$ by definition, we have that the period point of $X$ is also a point in the period domain of unpolarized Enriques surfaces. Hence, by the Torelli theorem for Enriques surfaces (\cite{Hor78a,Hor78b}, \cite{Nam85}), $X$ is the K3 cover of an Enriques surface $Y$. The involution $i$ fixes $\alpha(L_{10}(2))$, which is a 2-elementary lattice (i.e., its discriminant group is isomorphic to $(\bZ/2\bZ)^\ell$ for some $\ell\geq1$) and contains the ample divisor class $h$ (it is ample by the definition of polarized K3 surface). Moreover, the involution induced by $i$ on $H^2(X,\bZ)$ contains $T_X$ within its $(-1)$-eigenspace. Then, $i$ is unique by \cite[Lemma~5.3.3~(2)]{CDL25}. By construction, $Y$ is a $\ttbar$ Enriques surface.

Suppose now that $X$ is a very general $(M_R,h)$-polarized K3 surface. Then $\Pic(X)\cong M_R$. Since $M_R$ has rank at most $19$, the rank of the transcendental lattice of a very general such $X$ is at least $3$, hence the group of isometries of $T_X$ preserving $H^{2,0}(X)$ equals $\{\pm\mathrm{id}_{T_X}\}$ by \cite[Example~5.5]{Shi24}. It follows that $X$ covers a $\ttbar$-generic Enriques surfaces.
\end{proof}

\begin{proposition}
\label{prop:InvolutionInFamily}
Let $h\in\alpha(L_{10}(2))$ such that $h^2>0$. Let $T\to\Spec(\bC)$ be a scheme and consider a family $\sX \to T$ of $(M_R,h)$-polarized K3 surfaces. For every geometric fiber $\sX_t$ of the family, let $i_t$ be the Enriques involution of $\sX_t$ from Lemma~\ref{lem:polK3}. Then, the following hold:
\begin{itemize}

\item[(i)] Up to an \'etale base change of $T$, the involutions $i_t$ glue to an involution $\iota \colon \sX \to \sX$ acting fiberwisely;

\item[(ii)] The quotient $\sY\coloneqq \sX/\langle \iota\rangle\rightarrow T$ is a family of $\ttbar$ Enriques surfaces;

\item[(iii)] If $T$ is irreducible and the geometric generic fiber of $\sX$ is a very general $(M_R,h)$-polarized K3 surface, then the geometric generic fiber of $\sY$ is a $\ttbar$-generic Enriques surface.

\end{itemize}
\end{proposition}

\begin{proof}
The involution $\iota$ is defined by arguing as in the proof of \cite[Proposition~5.3.4]{CDL25}, starting from an involution $\sigma$ on $H^2(\sX,\bZ)$ satisfying $i_t^* = \sigma_t$, so that $\sigma$ induces the involution $i_t$ on each fiber $\sX_t$.
All the other statements in the lemma can be checked fiberwisely, and are then implied directly by Lemma~\ref{lem:polK3}.
\end{proof}

\begin{theorem}
\label{thm:boundLatticePol}
Let $Y$ be a $\ttbar$ Enriques surface. Then, we have that
\[
\nd(Y)\leq \nd(\overline{Y}),~\Fnd(Y)\leq \Fnd(\overline{Y}),~\textrm{and}~\Mnd(Y)\leq \Mnd(\overline{Y}),
\]
where $\overline{Y}$ is any $\ttbar$-generic Enriques surface.
\end{theorem}

\begin{proof}
Consider the universal family $\sX\to \sF_{M_R}$ of smooth $(M_R,h)$-polarized K3 surfaces. By the definition of Deligne--Mumford stack, there exists an \'etale surjective morphism $p\colon A \to \sF_{M_R}$ where $A$, called an \textit{atlas} for $\sF_{M_R}$, is a scheme. Let $\sX_A\rightarrow A$ be the pullback of $\sX\rightarrow\sF_{M_R}$ via $p$. 
Since the universal K3 covering $X\to Y$ is $(M_R,h)$-polarized, by the surjectivity of $p$ there exists $a\in A$ such that $(\sX_A)_a\cong X$.

Let $A'$ be an irreducible component of $A$ containing $a$. $A'$ is dominant on an irreducible component of $\sF_{M_R}$ by \cite[\href{https://stacks.math.columbia.edu/tag/0DR5}{Tag~0DR5}]{stacks-project}. This lemma applies to smooth morphisms from a scheme to a locally Noetherian algebraic stack, and indeed $p$ is smooth since it is \'etale (this implication reduces from the definitions \cite[Sections~101.33, 101.35]{stacks-project} to the analogous statement for algebraic spaces \cite[\href{https://stacks.math.columbia.edu/tag/04XX}{Tag~04XX}]{stacks-project}) and $\sF_{M_R}$ is locally Noetherian since it is smooth.

After possibly taking an \'etale cover of $A'$, we can apply Proposition~\ref{prop:InvolutionInFamily} and obtain a family $\sY \to A'$ with fibers $\ttbar$ Enriques surfaces. Since $A'$ dominates a component of $\sF_{M_R}$, again by Proposition~\ref{prop:InvolutionInFamily} we obtain that the geometric generic fiber $\overline{Y}$ of $\sY\to A'$ is $\ttbar$-generic. By construction, we have that $\sY_a\cong Y$.
By Corollary~\ref{cor:family}, we then have $\nd(Y)\leq \nd(\overline{Y})$, $\Fnd(Y)\leq \Fnd(\overline{Y})$, and $\Mnd(Y)\leq \Mnd(\overline{Y})$. 

Finally, the statement that $\overline{Y}$ can be taken to be any $\ttbar$-generic Enriques surface follows from the fact that
$\nd(\overline{Y})$, $\Fnd(\overline{Y})$, and $\Mnd(\overline{Y})$ only depend on $\ttbar$.
\end{proof}

Theorem \ref{thm:boundLatticePol} gives us a way to compute the non-degeneracy invariant of every surface that contains a configuration of smooth rational curves with dual graph as in \eqref{fig:configurations-smooth-rational-curves-BP-MLP1}~(a) or (b).

\begin{corollary}
\label{cor:ndBP}
Let $Y$ be an Enriques surface containing a configuration of smooth rational curves with dual graph as in \eqref{fig:configurations-smooth-rational-curves-BP-MLP1}~(a), then $\nd(Y)=4$.
\end{corollary}

\begin{proof}
First, we claim that $Y$ is an $(E_8,E_8)$ Enriques surface. Assuming this, let $\overline{Y}$ be any $(E_8,E_8)$-generic Enriques surface. Then, by Theorem~\ref{thm:boundLatticePol} we have that $\nd(Y)\leq\nd(\overline{Y})$. On the other hand, $\nd(\overline{Y})=4$ by \cite[Theorem~1.1~(1)]{MRS24Paper}, so we obtain the upper bound $\nd(Y)\leq4$. On the other hand, $Y$ admits a non-degenerate isotropic sequence of length $4$ (see for example \cite[Remark~4.3]{MMV22_NonDeg3}), which implies $\nd(Y)\geq4$.

We now prove that $Y$ is an $(E_8,E_8)$ Enriques surface. 
Let us consider the ADE-lattice $L\subseteq S_Y$ of ADE-type $E_8$ generated by the smooth rational curves corresponding to the filled-in vertices in Figure~\ref{fig:BarthPetersNum}. The lattice $L$ is primitive in $S_Y$ because it is unimodular, and hence $\overline{L}=E_8$. Let us fix an isometry $S_Y\cong L_{10}$ and denote by $R\subseteq L_{10}$ the image of $L$. Define $M_R$ as in Definition~\ref{def:tautaubar-generic-Enriques-surfaces}.

Let $\pi\colon X\rightarrow Y$ be the K3 cover. If $w_1,\ldots,w_8$ denote the elements of the chosen $E_8$ root basis of $L$, we have that $\pi^*w_i=w_i^++w_i^-$, and we can arrange the signs ``$\pm$'' so that $w_1^+,\ldots,w_8^+$ also form an $E_8$ root basis. Then, we can define an isometric embedding
\[
M_R \hookrightarrow \Pic(X)
\]
by sending $(v,0)$ with $v\in L_{10}(2)$ to $\pi^*v$, and every $(w_i, \pm w_i)/2$ to $w_i^\pm$. It follows from this construction that $Y$ is an $(E_8,E_8)$ Enriques surface.
\end{proof}

\begin{corollary}
\label{cor:ndmlp1}
Let $Y$ be an Enriques surface containing a configuration of smooth rational curves with dual graph as in \eqref{fig:configurations-smooth-rational-curves-BP-MLP1} (b), then $\nd(Y)=8$.
\end{corollary}

\begin{proof}
In this case, we have that $Y$ is a $(D_8,E_8)$ surface.  
Define $L$ to be the root lattice of ADE-type generated by the $D_8$ configuration corresponding to the filled-in vertices of Figure~\ref{fig:dual-graph-MLP1}. By Lemma~\ref{lem:primitive-closure-D8-lattice}, the primitive closure of $L$ is a copy of the $E_8$ lattice. Thus, via a construction analogous to the one carried out in the proof of Corollary~\ref{cor:ndBP}, we have that $Y$ is a $(D_8,E_8)$ Enriques surface. Then, by Theorem~\ref{thm:boundLatticePol}, $\nd(Y)\leq \nd(\overline{Y})$, where $\overline{Y}$ is any $(D_8,E_8)$-generic Enriques surface. We have $\nd(\overline{Y})=8$ by Theorem~\ref{thm:ndTauBarTau}, which will be proved in the next section. On the other hand, $\nd(Y)\geq 8$ by \cite[Proposition~6.2]{MRS22Paper}, hence $\nd(Y)=8$. 
\end{proof}

\begin{lemma}
\label{lem:primitive-closure-D8-lattice}
We retain the notation of the proof of Corollary~\ref{cor:ndmlp1}. The primitive closure of $L$ in $S_Y$ is isometric to $E_8$.
\end{lemma}

\begin{proof}
The vector $v:=\frac{1}{2}(e_1+e_4+e_6+e_8)$ (see Figure~\ref{fig:dual-graph-MLP1} for the labeling) is an element of $S_Y$ because it intersects every element of the basis of $S_Y$ in \cite[Proposition~6.1]{MRS22Paper} giving an integer. Moreover, $v\notin L$ and $2v\in L$: the latter is clear by definition of $L$, so let us explain the former.

Let $B$ be the Gram matrix of $L$ associated to the $\mathbb{Z}$-basis $(e_1,\ldots,e_8)$. The dual lattice $L^*$ is generated over $\mathbb{Z}$ by the rows of $B^{-1}$ and a set of generators for the discriminant group $A_L$ can be obtained by computing the Smith normal form of $B^{-1}$. Using the \texttt{SageMath} function \texttt{smith\_form()}, we obtain $M_1,M_2 \in\mathrm{SL}_8(\mathbb{Z})$ satisfying $M_1B^{-1}M_2=\mathrm{diag}\big(1,\ldots,1,\frac{1}{2},\frac{1}{2}\big)$. This not only illustrates that $A_L\cong(\mathbb{Z}/2\mathbb{Z})^2$, but also the rows of $M_1B^{-1}$ give an
alternative $\mathbb{Z}$-basis of $L^*$.

In particular, the last row vector of $M_1B^{-1}$ gives the following nonzero element of the discriminant group $A_L$:
\[
-\frac{1}{2}e_1-\frac{1}{2}e_4-\frac{1}{2}e_6+\frac{1}{2}e_8+L=v+L.
\]
Hence, $v\notin L$.

The above considerations imply that the primitive closure $\overline{L}\subseteq\mathrm{Num}(S_1)$ of $L$ contains $\langle L,v\rangle$. A the same time, $\langle L,v\rangle$ is an index $2$ overlattice of $L\cong D_8$. So, $\langle L,v\rangle$ is a negative definite unimodular lattice of rank $8$. These facts imply that $\overline{L}=\langle L,v\rangle$ and that $\overline{L}\cong E_8$.
\end{proof}

\begin{remark}\label{rmk:description}
    It is shown in \cite{MMV_ZeroEntropy} that Enriques surfaces containing a configuration of smooth rational curves as in \eqref{fig:configurations-smooth-rational-curves-BP-MLP1}~(a) are precisely the examples of Barth--Peters in \cite[Section~4]{BP83}, and the only Enriques surfaces of zero entropy in characteristic zero. Enriques surfaces containing a configuration as in \eqref{fig:configurations-smooth-rational-curves-BP-MLP1}~(b) instead are precisely those with eight disjoint smooth rational curves appearing in \cite[Example~1]{MLP02}. These two classes of surfaces are the only ones admitting cohomologically trivial automorphisms (see \cite{MN84}).
\end{remark}

\begin{remark}
\label{rmk:Enriques-surfaces-MLP2-nd}
We cannot apply the same strategy to compute the non-degeneracy invariant for surfaces of \cite[Example~2]{MLP02}. If $Y$ is an Enriques surface of this family, then $\nd(Y)\geq5$ by \cite[Proposition~6.5]{MRS22Paper}. On the other hand, we can prove that a very general such $Y$ is $(2D_4,D_8)$-generic, and hence it satisfies $\nd(Y)=10$ by \cite[Theorem~1.1]{MRS24Paper} (this is case $141$ in Brandhorst--Shimada's nomenclature). This is not enough to deduce the value of the non-degeneracy invariant for all the Enriques surfaces in the family.
\end{remark}

\begin{remark}
Arguing as in the proofs of Corollaries~\ref{cor:ndBP} and \ref{cor:ndmlp1}, one can show that
Enriques surfaces with finite automorphism group of type II are $(D_9,D_9)$ Enriques surface, and that Enriques surfaces in the Apery--Fermi family \cite{FV22} are $(E_7+A_2,E_7+A_2)$ Enriques surfaces.
\end{remark}

\begin{figure}
\noindent
\minipage{0.5\textwidth}
\centering
\begin{tikzpicture}[scale=0.65]

\node (F1) at (-135:2) [wdot, label=left:{}]{};
\node (F2) at (-90:2) [dot, label=below:{}]{};
\node (F3) at (-45:2) [dot, label=right:{}]{};
\node (F4) at (0:2) [dot, label=right:{}]{};
\node (F5) at (45:2) [dot, label=right:{}]{};
\node (F6) at (90:2) [dot, label=above:{}]{};
\node (F7) at (135:2) [dot, label=left:{}]{};
\node (F8) at (180:2) [dot, label=left:{}]{};
\node (L1) at (180:1) [wdot, label=above:{}]{};
\node (N) at (0:1) [dot, label=above:{}]{};

\draw[fiber] (F2)--(F3)--(F4)--(F5)--(F6)--(F7)--(F8) (F4)--(N);
\draw[]  (F8)--(L1) (F2)--(F1)--(F8);

\end{tikzpicture}
\caption{
An $E_8$ configuration.
}
\label{fig:BarthPetersNum}
\endminipage\hfill
\minipage{0.50\textwidth}
\centering
\begin{tikzpicture}[scale=0.65]

\node (1) at (90:2) [wdot]{};
\node (2) at (45:2) [dot,label=right:{$e_1$}]{};
\node (3) at (45:1) [dot,label=below:{$e_2$}]{};
\node (4) at (0:2) [dot,label=right:{$e_3$}]{};
\node (5) at (-45:2) [dot,label=right:{$e_4$}]{};
\node (6) at (-45:1) [wdot]{};
\node (7) at (-90:2) [dot,label=below:{$e_5$}]{};
\node (8) at (-135:2) [dot,label=left:{$e_6$}]{};
\node (9) at (-135:1) [wdot]{};
\node (10) at (180:2) [dot,label=left:{$e_7$}]{};
\node (11) at (135:2) [dot,label=left:{$e_8$}]{};
\node (12) at (135:1) [wdot]{};

\draw[fiber] (2)--(4)--(3) (4)--(5)--(7)--(8)--(10)--(11);
\draw[]  (3)--(1)--(2) (4)--(6)--(7)--(9)--(10)--(12)--(1)--(11);

\end{tikzpicture}
\caption{
A $D_8$ configuration.}
\label{fig:dual-graph-MLP1}
\endminipage
\end{figure}

\section{The non-degeneracy invariant of the \texorpdfstring{$(\tau,\overline{\tau})$}{(tau,bartau)}-generic Enriques surfaces}\label{sec:ttbarND}

In the current section we compute the Fano and Mukai non-degeneracy invariants for the class of $(\tau,\overline{\tau})$-generic Enriques surfaces, introduced in \cite{BS22} and defined in Section~\ref{ssec:ttbarDef}. This concludes the investigation started in \cite{MRS24Paper}. The main result of this section is Theorem~\ref{thm:ndTauBarTau}.

\subsection{Summary of Brandhorst and Shimada's results}
\label{BSConstruction}

We recall here some of the results of \cite{BS_Borcherds} and \cite{BS22} related to the $\ttbar$-generic Enriques surfaces. The Borcherd method for a K3 surface $X$ consists of embedding the lattice $S_X$ into $L_{26}$, the unique even unimodular hyperbolic lattice of rank $26$, and then studying $S_X$ through the action of $O(L_{26})$. In \cite{BS_Borcherds}, the authors extended this method to Enriques surfaces: first, Brandhorst and Shimada classified all primitive embeddings of $L_{10}(2)$ in $L_{26}$ up to the actions of $O(L_{26})$ and $O(L_{10}(2))$ (henceforth identified with $O(L_{10})$). The possible labels for the embeddings $\iota\colon L_{10}(2) \to L_{26}$ are listed in the column \texttt{name} in \cite[Table~1.1]{BS_Borcherds}.

Then, fixed $\iota$ not labeled \texttt{infty}, they define a so-called \textit{induced chamber} $D\subseteq L_{10}\otimes\mathbb{R}$ \cite[Definition~2.4]{BS22}. Every induced chamber $D$ has finitely many walls which are defined by $(-2)$-vectors of $L_{10}$. In \cite{BS_Borcherdsdata} the walls are listed explicitly for a particular choice of induced chamber denoted by $D_0$.

As a result, we have the following data for each $(\tau,\overline{\tau})$-generic Enriques surface $Y$:
\begin{itemize}
\item a primitive embedding $S_Y(2)\xrightarrow{\lambda_Y} L_{10}(2) \xrightarrow{\iota} L_{26}$,
\item a finite list of walls of a chamber $D_0 \subseteq \lambda_Y(\Nef(Y))$ (depending only on $\iota$), and
\item a finite list $\{g_0=\mathrm{id},g_1,\ldots,g_M\}$ of isometries $g_j \in O(L_{10})$, $j=0,\ldots,M$,
\end{itemize}
such that $F_0\coloneqq \lambda_Y^{-1}\left(\bigcup_{j=0}^Mg_j(D_0)\right)$ contains a fundamental domain for the action of $\aut(Y)$ on $\Nef(Y)$. 
For each $Y$ as above, \cite{BS22} also provides generators for the group $\aut(Y)$, as well as representatives for the orbits of $\aut(Y)$ acting on the sets of smooth rational curves and of elliptic fibrations on $Y$. This is done via the Borcherd method, see \cite[Section~4]{BS22} for details. In particular, for every $Y$ we have the following data:
\begin{itemize}
    \item a finite set $\sR_{\mathrm{temp}}$ of smooth rational curves;
    \item a finite set $\sE_{\mathrm{temp}}$ of elliptic fibrations;
    \item a finite subset $\sH\subseteq \aut(Y)$.
\end{itemize}
These lists are specified in the records \texttt{Rats.Ratstemp}, \texttt{Ells.Ellstemp}, and \texttt{Autrec.HHH} of \cite{SB20data}, see \cite[Example~3.5]{MRS24Paper}.
The label of $\iota$ is specified in the entry \texttt{irec} of \cite[Table~1.1]{BS22}.

\subsection{Enriques surfaces with finite automorphism group: primitive embeddings}
\label{subs:FiniteAuto}
Enriques surfaces with finite automorphism group were classified by Kond\=o in \cite{Kon86}. These come into seven families, labeled by $\mathrm{I},\mathrm{II},\ldots,\mathrm{VII}$, which we will refer to as \emph{Kond\=o types}. If $Z$ is one of such surfaces, by \cite[Table~1.2]{BS_Borcherds} there exists a suitable marking $\lambda_Z\colon S_Z\cong L_{10}$ and a primitive embedding $\iota\colon L_{10}(2) \hookrightarrow L_{26}$ such that $\lambda_Z(\Nef(Z))=D_0$. 
The correspondence between the Kond\=o and embedding types is as follows:
\begin{equation}
\label{table:Irec_Kondo}
\begin{tabular}{c||c|c|c|c|c|c|c}
Kond\=o type & $\mathrm{I}$ & $\mathrm{II}$ & $\mathrm{III}$ & $\mathrm{IV}$ & $\mathrm{V}$ & $\mathrm{VI}$ & $\mathrm{VII}$\\
\hline
Embedding type & \texttt{12A}  & \texttt{12B}  &\texttt{20B}  &\texttt{20F}  & \texttt{20A} & \texttt{20E} & \texttt{20C},~\texttt{20D}
\end{tabular}
\end{equation}
(As already clear from the table, the Enriques surface with finite automorphism group of type $\mathrm{VII}$ corresponds to two possible embedding types with isomorphic induced chambers \cite[Table~1.2]{BS_Borcherds}.) Recall that for an Enriques surface $Y$ we write $\sfF_Y$ and $\sfM_Y$ for the sets of all its numerical Fano (respectively, Mukai) polarizations (see Definition~\ref{def:set-of-polarizations}).
If $Y$ is $(\tau,\overline{\tau})$-generic, we denote by $\sfF_{Y}^0$ (resp. $\sfM_{Y}^0$) the set of all numerical Fano (resp. Mukai) polarizations contained in $\lambda_Y^{-1}(D_0)\subseteq S_Y$. Namely, $\sfF_Y^0:=\sfF_Y\cap\lambda_Y^{-1}(D_0)$ and $\sfM_Y^0:=\sfM_Y\cap\lambda_Y^{-1}(D_0)$.

\begin{lemma}\label{lem_FKondoF0}
Let $Y$ be a $(\tau,\overline{\tau})$-generic Enriques surface whose embedding type appears in \eqref{table:Irec_Kondo}. Let $Z$ be an Enriques surface with finite automorphism group with the same embedding type. Then, the isometry $\lambda:=\lambda_Z^{-1}\circ\lambda_Y\colon S_Y\rightarrow L_{10}\rightarrow S_Z$ satisfies $\lambda(\sfF_Y^0)=\sfF_Z$ and $\lambda(\sfM_Y^0)=\sfM_Z$.
\end{lemma}

\begin{proof}
Let $h \in \sfF_Y^0=\sfF_Y\cap\lambda_Y^{-1}(D_0)$. Then, $h'\coloneqq \lambda(h) =\lambda_Z^{-1}\circ\lambda_Y(h)\in \Nef(Z)$. Moreover, we have that $(h')^2 = h^2 = 10$ and $\Phi(h')=\Phi(h)=3$ because $\lambda$ is an isometry. This implies that $h'\in\sfF_Z$. Conversely, if $h'\in\sfF_Z$, then $\lambda^{-1}(h')=\lambda_Y^{-1}\circ\lambda_Z(h')\in\lambda_Y^{-1}(D_0)$. Finally, $\lambda^{-1}(h')\in\sfF_Y$ by an argument analogous to the one above. Hence, $\lambda^{-1}(h')\in\sfF_Y^0$. A similar proof holds for $\lambda(\sfM_Y^0)=\sfM_Z$.
\end{proof}

\subsection{Main result}
\label{sec:mainttb}

We now compute $\nd$, $\Fnd$, and $\Mnd$ for the $\ttbar$-generic Enriques surfaces.

\begin{theorem}
\label{thm:ndTauBarTau}
Let $Y_i$ denote the $i$-th $\ttbar$-generic Enriques surface of \cite[Table~1]{BS22}. 
Then $\nd(Y_i)=\Fnd(Y_{i})=10$ and $\Mnd(Y_i)=9$ except for the cases below, in which the three invariants coincide:
\begin{equation}\label{eq:14cases}
\setlength{\arraycolsep}{5pt}
\begin{aligned}
&\begin{array}{cccc}
\text{no.} & \text{irec} & (\tau,\overline{\tau}) & \nd \\ \hline
145 & 12\text{A} & (E_8,E_8)                    & 4 \\
172 & 12\text{A} & (E_8+A_1,E_8+A_1)            & 4 \\ \hline
85  & 20\text{A} & (E_7,E_7)                    & 7 \\
122 & 20\text{A} & (E_7+A_1,E_7+A_1)            & 7 \\
123 & 20\text{A} & (E_7+A_1,E_8)                & 7\\
159 & 20\text{A} & (E_7+2A_1,E_8+A_1)           & 7 \\
176 & 20\text{A} & (E_7+A_2,E_7+A_2)            & 7 
\end{array}
&&
\begin{array}{cccc}
\text{no.} & \text{irec} & (\tau,\overline{\tau}) & \nd \\ \hline
143 & 12\text{B} & (D_8,D_8)                    & 8 \\
184 & 12\text{B} & (D_9,D_9)                    & 7 \\ \hline
84  & 20\text{B} & (D_7,D_7)                    & 9 \\
121 & 20\text{B} & (D_7+A_1,D_7+A_1)            & 9 \\
144 & 20\text{B} & (D_8,E_8)                    & 8 \\
158 & 20\text{B} & (D_7+2A_1,D_9)               & 9 \\
171 & 20\text{B} & (D_8+A_1,E_8+A_1)            & 8
\end{array}
\end{aligned}
\end{equation}
(the first column ``\emph{no.}'' refers to the list in \cite{BS22}).
\end{theorem}

\begin{proof}[Proof of Theorem~\ref{thm:ndTauBarTau}]
We have that $\nd(Y)=\max\{\Fnd(Y),\Mnd(Y)\}$ by Corollary~\ref{cor:max}. Then, the conclusion follows by combining Proposition~\ref{prop_LowerBounds}, Proposition~\ref{prop_Fnd}, and Proposition~\ref{prop_Mnd} below. 
\end{proof}

\begin{remark}
Some values of $\nd$ listed in Theorem~\ref{thm:ndTauBarTau} were previously known. We have that $Y_1$ is general nodal, and $Y_{172},Y_{184}$ have finite automorphism group (see \cite[Remark~4.1]{MRS24Paper}). In \cite[Section~5]{MRS24Paper}, we computed $\nd(Y_{145})$ through a detailed description of its automorphism group, and produced lower bounds for $\nd(Y_i)$ when $i\neq 1,172,184$. In particular, in all the cases not listed in \eqref{eq:14cases}, \cite{MRS24Paper} gives $\nd(Y_i)\geq10$, and hence $\nd(Y_i)=10$.
\end{remark}

We now prove Propositions~\ref{prop_LowerBounds}, \ref{prop_Fnd}, and \ref{prop_Mnd}. Lower bounds for $\Fnd$ and $\Mnd$ in Theorem~\ref{thm:ndTauBarTau} can be obtained by exhibiting appropriate isotropic sequences. The proof is similar to the one of \cite[Theorem~1.1]{MRS24Paper} for $\nd$, but it rests on a more refined search where we distinguished sequences of length $10$ and sequences in $\One$. In what follows, we define $N_i$ to be the value of $\nd(Y_i)$ claimed in \eqref{eq:14cases}.

\begin{proposition}
\label{prop_LowerBounds}
The following hold:
\begin{itemize}
    \item[(i)] For all surfaces $Y_i$ listed in \eqref{eq:14cases}, $\Fnd(Y_i)\geq N_i$ and $\Mnd(Y_i)\geq N_i$. In particular, we have that $\Mnd(Y_{84})=\Mnd(Y_{121})=\Mnd(Y_{158})=9$;
    \item[(ii)] For all surfaces $Y_i$ not appearing in \eqref{eq:14cases}, $\Fnd(Y_i)= 10$ and $\Mnd(Y_i)= 9$.
\end{itemize}
\end{proposition}

\begin{proof}
Let $Y_i$ be as in \eqref{eq:14cases}. To show that $\Fnd(Y_i)\geq N_i$ (resp. $\Mnd(Y_i)\geq N_i$), it is enough to exhibit an explicit length $10$ (resp. length $9$ non-extensible) isotropic sequence with non degeneracy $N_i$. For $i\neq84,121,158$, we have that $N_i\neq9$ and we know from \cite{MRS24Paper} that $\nd(Y_i)\geq N_i$, hence also $\Fnd(Y_i),\Mnd(Y_i)\geq N_i$. For $i\in\{84,121,158\}$, appropriate isotropic sequences can be found in \cite{DataNew}. Since $N_i=9$, we obtain $\Fnd(Y_i)\geq9$ and $\Mnd(Y_i)=9$. This proves (i). For part (ii), suppose that $Y_i$ does not appear in \eqref{eq:14cases}. By \cite[Theorem~1.1]{MRS24Paper} we have that $\nd(Y_i)=10$, which immediately gives $\Fnd(Y_i)=10$ by definition. Finally, to show that $\Mnd(Y_i)=9$, appropriate length $9$ non-extensible isotropic sequences of non-degeneracy $9$ can be found in \cite{DataNew}.
\end{proof}

\begin{proposition}
\label{prop_Fnd}
For all surfaces $Y_i$ listed in \eqref{eq:14cases}, $\Fnd(Y_{i})=N_i$.
\end{proposition}

\begin{proof}
By \cite[Theorem~1.1~(1)]{MRS24Paper}, we have that $\nd(Y_{145})=4$. The surfaces $Y_{172}$ and $Y_{184}$ have finite automorphism groups, for which $\nd$ is known: $\nd(Y_{172})=4$ and  $\nd(Y_{184})=7$, see \cite[Remark~4.1]{MRS24Paper}. Hence, $\Fnd(Y_{145})=4$, $\Fnd(Y_{172})=4$, and $\Fnd(Y_{184})=7$ by Proposition~\ref{prop:casesNdFndMnd}.

For the remaining $11$ cases, we need to show that $\Fnd(Y_i)\leq N_i$.
For simplicity, we drop the subscript $i$ and write $Y=Y_i$, $N=N_i$.
As introduced in Section~\ref{BSConstruction}, there are $g_0=\mathrm{id},\ldots,g_M\in O(L_{10})$ and a chamber $D_0$ in $L_{10} \otimes \bR$ such that $\lambda_Y^{-1} \left(\bigcup_{j=0}^Mg_j(D_0)\right)$ contains a fundamental domain for the action of $\aut(Y)$ on $\Nef(Y)$. Let $\sfF_Y^0$ be the set of numerical Fano polarizations contained in $\lambda_Y^{-1}(D_0)$. Then $\sfF_Y'\coloneqq \bigcup_{j=1}^Mg_j'(\sfF_Y^0)$, where $g'_j \coloneqq \lambda_Y^{-1}\circ g_j\circ\lambda_Y$, is a finite set of representatives for the orbits of the action of $\aut(Y)$ on $\sfF_Y$ as in Definition~\ref{def:set-of-polarizations}.

The \texttt{irec} label of $Y$ is \texttt{12B} for $i=143$, \texttt{20A} for $i=85,122,123,159,176$, and \texttt{20B} for $i=84,121,144,158,171$. Then, by Lemma~\ref{lem_FKondoF0}, the set $\sfF_Y^0$ matches $\sfF_Z$ for some Enriques surface $Z$ with finite automorphism group as in \eqref{table:Irec_Kondo} (the sets $\sfF_Z$ were computed explicitly in Section~\ref{sec:num-Fano-pol-En-surf-with-finite-aut}).

For each $h \in \sfF_Y'$, we find $R_1,\ldots, R_{10-N}$ distinct smooth rational curves in $\sR_{\mathrm{temp}}$, or possibly in $\sR_{\mathrm{temp}}$ acted upon by $\sH$, such that $h\cdot R_k =0$ for $k=1,\ldots,10-N$: the list of such $R_k$ for each $h\in\sfF_Y'$ is available at \cite{DataNew}. Then $\Fnd(Y)\leq N$ by Theorem~\ref{cor_ndUpperBound}~(1).
\end{proof}

\begin{remark}
\label{rem:suggestionSimon}
In the proof of Proposition~\ref{prop_Fnd}, an alternative approach to construct a set $\sfF_Y'$ was suggested to us by Simon Brandhorst. This method does not rely on the Enriques surfaces with finite autormorphism group (Section~\ref{subs:FiniteAuto}), and it therefore applies also to the $\ttbar$-generic Enriques surfaces whose \texttt{irec} does not appear in \eqref{table:Irec_Kondo}. 

As in \cite[Section~4.2]{BS22}, $D_0$ is tessellated by Vinberg chambers. If $(e_1,\ldots,e_{10})$ is a standard root basis of $L_{10}$, the chamber $C$ defined by $\langle x,e_i\rangle\geq 0$ is a Vinberg chamber and it contains only one vector $v$ of degree $10$ with $\Phi(v)=3$, see \cite[Corollary~1.5.4]{CDL25}. A set $\sfF_Y'$ can then be obtained by recursively reflecting $v$ across the walls of the chosen Vinberg chambers in $D_0$. We checked the lists in \cite{DataNew} with this alternative method.
\end{remark}

\begin{proposition}
\label{prop_Mnd}
For all surfaces $Y_i$ listed in \eqref{eq:14cases}, $\Mnd(Y_{i})=N_i$.
\end{proposition}

\begin{proof}
Suppose $N_i\neq8$. Then, combining Proposition~\ref{prop:casesNdFndMnd}, Proposition~\ref{prop_LowerBounds}~(i), and Proposition~\ref{prop_Fnd}, we have that $N_i\leq\Mnd(Y_i)=\Fnd(Y_i)=N_i$. Next, we consider the case $N_i=8$, which occurs for $i=143,144,171$. For simplicity of notation, let $Y=Y_i$.

For each such $Y$, we know from Proposition~\ref{prop_Fnd} that $\Fnd(Y)=8$, hence $\Mnd(Y)=8$ or $9$ by Proposition~\ref{prop:casesNdFndMnd}. So, we need to prove that $\Mnd(Y)\leq8$. We do this via Theorem~\ref{cor_ndUpperBound}~(2).

First, we construct a finite set $\sfM_Y'$ of representatives for the orbits of the action of $\aut(Y)$ on $\sfM_{Y}$ as we did in the proof of Proposition~\ref{prop_Fnd}. This is done using the sets $\sfM^0_Y$ and $\sfM_Z$ instead of $\sfF_Y^0$ and $\sfF_Z$. By Theorem~\ref{cor_ndUpperBound}~(2) applied to $\sfM_Y'$, we have that $\Mnd(Y)\leq 8$ if and only if for each $v \in \sfM_Y'$ with associated canonical isotropic sequence $\bfg = (g_1,\ldots,g_9)$, there exists a smooth rational curve $R$ such that $v \cdot R=0$ and $R\notin \langle g_1,\ldots, g_9 \rangle^\perp$. So, given $v \in \sfM_Y'$, we search for such an $R$ in $\sR_{\mathrm{temp}}$, or possibly in $\sR_{\mathrm{temp}}$ acted upon by $\sH$.
One of the following two scenarios occur (the computational data is available at \cite{DataNew}):
\begin{enumerate}
\item We find two distinct smooth rational curves $R_1,R_2$ orthogonal to $v$. Then, by Lemma~\ref{lem_uniqueOrt} there exists $R\in\{R_1,R_2\}$ such that $g_j \cdot R\neq 0$ for some $j=1,\ldots,9$. This is because at most one smooth rational curve is orthogonal to all $g_j$. So we are done.
\item We find a single smooth rational curve $R$ orthogonal to $v$. In this case, we run a secondary search: by Lemma~\ref{lem_UniqueCanonicalMukai}~(1), the vectors $g_i$ are the only effective isotropic classes satisfying $v\cdot g_i = 4$. Thus, we search in $\sE_{\mathrm{temp}}$ (and its $\sH$-translates) for vectors $f$ satisfying $v\cdot f=4$. We are able to find vectors $\{f_1,\ldots,f_8\}$ such that $\bfg = (f_1,\ldots,f_8,f_1+R)$, up to reordering. From this we can see that $R$ has the required properties.
\end{enumerate}
\end{proof}

\begin{remark}
There is an alternative (but less uniform) way to prove Theorem~\ref{thm:ndTauBarTau}.
Indeed, we can compute $\nd$, $\Fnd$, and $\Mnd$ of the $\ttbar$-generic Enriques surfaces with $\ttbar=(E_7,E_7)$, $(D_7,D_7)$, $(D_8,D_8)$, and $(D_8,E_8)$ in the same way we did above. After this, the remaining cases follow by Theorem~\ref{thm:boundLatticePol}, since:
\begin{itemize}
        \item the $\ttbar$-generic Enriques surfaces with embedding type \texttt{20A} are $(E_7,E_7)$ Enriques  surfaces;
        \item the $(D_7+A_1,D_7+A_1)$-generic and $(D_7+2A_1,D_9)$-generic Enriques surfaces are $(D_7,D_7)$ Enriques surfaces;
        \item the $(D_8+A_1,E_8+A_1)$-generic Enriques surfaces are $(D_8,E_8)$ Enriques surfaces.
\end{itemize}
\end{remark}

\subsection{Non-degeneracy invariants and birational models: examples and remarks}\label{sec:Mukai}

We conclude the paper with some examples and open questions. The examples highlight a somewhat unexpected behavior of certain Mukai polarizations and their corresponding birational models.

Let $v$ be a numerical Mukai polarization with associated canonical sequence $\bfg=(g_1,\ldots,g_9)$ (Lemma~\ref{lem_MukaiPolCanonicalSeq}). If $v$ is ample, then all the $g_i$ are nef: indeed, if a smooth rational curve $R$ appears in $\bfg$ then $v\cdot R=0$. 
The converse does not hold: we exhibit in Example~\ref{exa_nefMukai} an Enriques surface $Y$ with a numerical Mukai polarization $v$ which is not ample and whose associated canonical isotropic sequence $\bfg = (g_1,\ldots, g_9)$ consists of nef classes. This happens precisely when the $(-2)$-class discussed in Lemma~\ref{lem_uniqueOrt} is effective.

\begin{example}[A singular non-degenerate Mukai model]\label{exa_nefMukai}
    Let $Y$ the surface VI in Kond\=o's classification \cite{Kon86}. Take any of the ten saturated sequences of length $9$ of type $6 \times (\widetilde{A}_1^{\hf} + \widetilde{A}_2^{\f} + \widetilde{A}_5^{\f}),  3 \times (\widetilde{A}_3^{\f} + \widetilde{D}_5^{\f})$, see the fourth row of \cite[Table~9]{MRS22Paper}. These give rise to Mukai polarizations which are not ample. This claim can be verified by checking the computational data \cite{DataNew}.
\end{example}

Next, we exhibit an Enriques surface which admits an ample Mukai polarization, but no ample Fano polarizations.

\begin{example}[A smooth Mukai model without smooth Fano models]\label{ex_Fnd=Mnd=9}
Consider a $(D_7+2A_1,D_9)$-generic Enriques surface $Y=Y_{158}$. We know from Theorem~\ref{thm:ndTauBarTau} that $\nd(Y)=\Fnd(Y)=9$, so there is no smooth Fano model. Consider the following saturated sequence $\bfg = (g_1,\ldots, g_9)$ of length $9$:
\begin{align*}
&g_1 := \frac{1}{2}(R_{0} + R_{2}), \quad g_2 := \frac{1}{2}(R_{2} + R_{16}), \quad g_3 := \frac{1}{2}(R_{2} + R_{2}(H_0)^{-1}), \\ &g_4 := \frac{1}{2}(R_{3} + R_{4}), \quad g_5 := \frac{1}{2}(R_{3} + R_{12}), \quad g_6 := \frac{1}{2}(R_{2} + R_{14}), \\
&g_7 := \frac{1}{2}(R_{2} + R_{8}), \quad g_8:= \frac{1}{2}(R_1+R_6+2R_7+R_8+R_{15}),\\
&g_9 := \frac{1}{2}(R_5+R_6+R_8+R_9+2R_{11}).
\end{align*}
The indexing comes from the specific ordering in the data \texttt{Rats.Ratstemp}, and \texttt{Autrec.HHH}.
If $R$ is a $(-2)$-curve such that $v\cdot R = 0$, by Lemma~\ref{lem_UniqueCanonicalMukai}~(2) we must have $g_j\cdot R=0$ for all $j=1,\ldots,9$ since $\bfg$ is non-degenerate. However, there are no $(-2)$-curves in $H = \langle g_1,\ldots,g_9 \rangle^\perp$: indeed there are only two $(-2)$-vectors in $H$. Denote them by $s_1$ and $s_2=-s_1$. It is not hard to exhibit, for $i=1,2$, an elliptic fibration with class $f_i$ such that $f_i\cdot s_i< 0$, which shows that the $s_i$ are not effective, by \cite[Remark~III.5]{Beauville}. The details of this computation can be found in \cite{DataNew}.
\end{example}

The above example shows that the existence of an ample Mukai polarization on $Y$ does not imply that of an ample Fano polarization. However, we may weaken the question, and ask the following.

\begin{question} 
Does the existence of an ample Mukai polarization on $Y$ imply $\Fnd(Y)\geq 9$? We may further weaken the question: does $\Mnd(Y)=9$ imply $\Fnd(Y)\geq 9$?
\end{question}

An affirmative answer to the latter is equivalent to saying that Case~(iv) in Proposition~\ref{prop:casesNdFndMnd} does not occur. 
We point out that by Theorem~\ref{thm:ndTauBarTau}, Cases~(ii), (iii), (iv) of Proposition~\ref{prop:casesNdFndMnd} do not occur for $\ttbar$-generic Enriques surfaces. We therefore ask the following question.

\begin{question}
Are there examples of Enriques surfaces which realize Cases~(ii), (iii), (iv) of Proposition~\ref{prop:casesNdFndMnd}?
\end{question}

It was already observed in \cite[Remark~4.3]{MRS24Paper} that no $\ttbar$-generic Enriques surface has non-degeneracy invariant $5$ or $6$. Across all characteristics, examples are known of Enriques surfaces with non-degeneracy invariant equal to $1,2,3,4,7,8,9,10$ (see \cite{MMV22_NonDeg3} for the first three cases). It is then natural to ask the following.

\begin{question}
Are there examples of Enriques surfaces with non-degeneracy invariant $5$ or $6$?
\end{question}

\end{document}